\documentclass{compositio}
\usepackage{amsmath,amssymb,amsfonts,amscd}
\usepackage[mathscr]{eucal}
\usepackage{tikz}
\usepackage{verbatim}
\usepackage[colorlinks,pagebackref]{hyperref}
\usepackage{rotating}

\usepackage{color}
\theoremstyle{plain}
\newtheorem{thm}{Theorem}[section]
\newtheorem{lem}[thm]{Lemma}

\newtheorem{prop}[thm]{Proposition}

\newtheorem*{thm*}{Theorem}

\newtheorem{sublem}[equation]{Lemma}
\newtheorem{subcor}[equation]{Corollary}
\newtheorem{subprop}[equation]{Proposition}
\newtheorem{subthm}[equation]{Theorem}

\theoremstyle{definition}

\newtheorem{cosa}[thm]{}
\newtheorem{subcosa}[equation]{}

\newtheorem{subex}[equation]{Example}

\theoremstyle{remark}

\newtheorem{subrem}[equation]{Remark}

\newtheorem{subrems*}{Remarks.}

\numberwithin{equation}{thm}

\newcommand{\D}{\boldsymbol{\mathsf{D}}}
\newcommand{\M}{\boldsymbol{\mathsf{M}}}

\newcommand{\LL}{\mathsf L}
\newcommand{\R}{\mathsf R}

\newcommand{\fp}{{\mathfrak p}}

\newcommand{\sE}{\mathscr{E}}

\newcommand{\sL}{\mathscr{L}}
\newcommand{\sM}{\mathscr{M}}
\newcommand{\sO}{\mathscr{O}}

\newcommand{\Dc}{\boldsymbol{\mathsf{c}}}

\newcommand{\1}{\mathbf{1}}

\newcommand{\op}{{\mathsf o\mathsf p}}

\newcommand{\ZZ}{\mathbb Z}

 \newcommand{\Rf}{\R f^{}_{\<\<*}}

\newcommand{\fst}{{f^{}_{\<\<*}}}

\newcommand{\cc}{\mathsf{c}}

\newcommand{\qc}{\mathsf{qc}}

\newcommand{\Dqc}{\D_{\mathsf{qc}}}

\newcommand\Dqcpl{\D_\qc^{\lift.95,\text{\cmt\char'053},}}

\newcommand\Dcpl{\D_\cc^{\lift.95,\text{\cmt\char'053},}}

\font\cmt=cmtex10

\newcommand{\CH}{\mathcal H}

\newcommand{\bpic}{\begin{tikzpicture}}
\newcommand{\epic}{\end{tikzpicture}}

\newcommand{\Otimes}[1]{\otimes^\LL_{#1}}

\newcommand{\sHom}{\CH om}
\newcommand{\RHqc}[1]{\R\>\sHom_{#1}^{\<\qc}}
\newcommand{\set}{\!:=}

\newcommand{\sX}{{\<\<X}}

\newcommand{\sst}{\scriptstyle}
\newcommand{\sss}{\scriptscriptstyle}

\newcommand{\smallcirc}{{\>\>\lift1,\sst{\circ},\,}}

\newcommand{\<}{\mkern-1mu}
\renewcommand{\>}{\mkern1mu}
\newcommand{\va}[1]{\vspace{#1pt}}

\newcommand{\kf}{\kern.5pt}

\def\lift#1,#2,{\vbox to 0pt{\vskip-#1 ex\hbox{$\scriptstyle #2$}\vss}}

\newcommand{\OX}{\mathcal O_{\<\<X}}
\newcommand{\OY}{\mathcal O_Y}
\newcommand{\OZ}{\mathcal O_{\<Z}}

\newcommand{\OW}{\mathcal O_W}

\newcommand{\fundamentalclass}[1]{{\boldsymbol{\mathsf{c}}}_{#1}} 
\newcommand{\fundamentalclassa}[1]{{\boldsymbol{\mathsf{a}}}_{#1}}

\newcommand{\fundamentalclassb}[1]{{\boldsymbol{\mathsf{b}}}_{#1}}

\newcommand{\circled}[1]{\textcircled{\scriptsize{#1}}}

\newcommand{\lto}{\longrightarrow}
\newcommand{\xto}{\xrightarrow}

\newcommand\iso{{\mkern8mu\longrightarrow \mkern-25.5mu{}^\sim\mkern17mu}}
\newcommand\osi{{\mkern8mu\longleftarrow \mkern-23.5mu{}^\sim\mkern17mu}}

\DeclareMathOperator{\spec}{Spec}
\DeclareMathOperator{\Hom}{Hom}

\DeclareMathOperator{\id}{id}

\DeclareMathOperator{\via}{{\textup{via}}}

\DeclareMathOperator{\supp}{supp}
\DeclareMathOperator{\Supp}{Supp}

\def\Iso{\vbox to 0pt{\vss\hbox{$\widetilde{\phantom{nn}}$}\vskip-7pt}}

\newcommand{\vG}{\varGamma}

\newcommand{\env}[1]{{#1}^{\mathsf e}}

\def\cA #1; #2;{\cite[p.\,#1, #2]{A}}
\def\cT #1; #2;{\cite[p.\,#1, #2]{T}}

\def\lift#1,#2,{\vbox to 0pt{\vskip-#1 ex\hbox{$\scriptstyle #2$}\vss}}

\def\drlm#1{\underset{\vtop{\vskip-4.2pt\hbox to 14pt{\rightarrowfill} \vskip-10pt\hbox{$\scriptstyle \ #1$}}}\to\lim\,}

\def\dirlm#1{\lim\hskip-1.65em\lower1.37ex
       \hbox{\smash[b]{$
                   \underset{\lift 1.37,
                                         {\hbox to 0pt{\hss$\scriptscriptstyle#1$\hss}},
                                  }
                     {\:\hbox to 1.37em {\rightarrowfill}}
               $} }                      
     \!\<}

\begin{document}

\title[Relation between two twisted inverse image pseudofunctors]{Relation between two twisted inverse image pseudofunctors in duality theory}

\author[S.\,B.\,Iyengar]{Srikanth B.~Iyengar} 
\email{s.b.iyengar@unl.edu}
\address{Department of Mathematics,  University of Nebraska, Lincoln, NE 68588, U.S.A.}

\author[J.\,Lipman]{Joseph Lipman} 
\email{jlipman@purdue.edu}
\address{Department of Mathematics,  Purdue University, West Lafayette IN 47907, U.S.A.}

\author[A.\,Neeman]{Amnon Neeman} 
\email{Amnon.Neeman@anu.edu.au}
\address{Centre for Mathematics and its Applications, Mathematical Sciences Institute
Australian National University, Canberra, ACT 0200, Australia.}

\thanks{This article is based on work supported by the National Science Foundation under Grant No. 0932078000, while the authors were in residence at the Mathematical Sciences Research Institute in Berkeley, California, during the Spring semester of 2013.  The first author was partly supported by NSF  grant DMS 1201889 and a Simons Fellowship.}
  
\keywords{Grothendieck duality, twisted inverse image pseudofunctors,
Hochschild derived functors,  relative perfection, relative dualizing complex,
fundamental class.}

\subjclass[2010]{Primary: 14F05, 13D09. Secondary:  13D03}

%\date{\today}

\begin{abstract}  Grothendieck duality theory assigns to essentially-finite-type maps~
$f$ of noetherian schemes a pseudofunctor $f^\times$ 
right-adjoint to $\R\fst$, and a pseudofunctor $f^!$ agreeing with $f^\times$ when
$f$ is proper, but equal to the usual inverse image $f^*$ when $f$ is \'etale. We define and study a canonical map from the first pseudofunctor to the second. This map behaves well with respect to flat base change, and is taken to an isomorphism by ``compactly supported" versions of standard derived functors. Concrete realizations are described, for instance for maps of affine schemes. Applications include proofs of reduction theorems for Hochschild homology and cohomology, and of a remarkable formula for the fundamental class of 
a flat map of affine schemes.
 \end{abstract}

\maketitle

\section*{Introduction}
The relation in the title is given by a canonical pseudofunctorial map $\psi\colon(-)^\times\to (-)^!$ between ``twisted inverse image" pseudofunctors with which Grothendieck duality theory is concerned. These pseudofunctors on the category $\sE$ of  essentially-finite\kf-type separated maps of noetherian schemes take values in bounded-below derived categories  of complexes with 
quasi-coherent homology, see \ref{^!} and~\ref{^times}.\va{.6}  The map $\psi$, derived from the pseudofunctorial ``fake unit map" $\id\to(-)^!\smallcirc\R(-)_*$ of Proposition~\ref{fake unit}, is specified\ in Corollary~\ref{relation}.  A number of concrete examples appear in 
\S\ref{section:examples}. For instance, if  $f$ is a map in~$\sE$, then $\psi(f)$ is an isomorphism if $f$ is proper; but if $f$ is, say, an open immersion, so that 
$f^!$ is the usual inverse image functor~$f^*$ whereas $f^\times$ is right-adjoint to 
$\R\fst\>$, then $\psi(f)$ is usually quite far from being an isomorphism (see e.g., ~\ref{affine locimm}, ~\ref{varGM} and ~\ref{affine/k}).

After some preliminaries are covered in \S\ref{prelims}, the definition of the pseudofunctorial map $\psi$ is worked out at the beginning of \S\ref{basic map}.
Its good behavior with respect to flat base change is given by Proposition~\ref{base change}.

The rest of Section~\ref{basic map} shows that under suitable ``compact support" conditions,  various operations from duality theory take $\psi$ to an isomorphism.
To wit:

Let $\Dqc(X)$ be the derived category of $\OX$-complexes with 
quasi-coherent homology, and let $\RHqc{\sX}(-,-)$ be
%, as in \S\ref{RHom^qc},
the internal hom in the closed category $\Dqc(X)$ (\S1.5).
Proposition~\ref{Gampsi} says: 

\noindent\emph{If\/ $f\colon X\to Y$ is a map in\/ $\sE$,
if\/ $W$ is a union of closed subsets of\/~$X$  to each of which the restriction 
of\/~$f$ is proper, and if\/ $E\in\Dqc(X)$ has support contained in\/~$W\<,$ then each of the  functors\/ $\R\vG^{}_{\!W}(-),$ $E\Otimes{\sX}(-)$ and\/  $\RHqc{\sX}(E,-)$  takes the map\/ $\psi(f)\colon f^\times\to f^!$ to an isomorphism.}  

The proof uses properties of a bijection between subsets of $X$ and ``localizing tensor ideals" in $\Dqc(X)$, reviewed\- in Appendix~\ref{Support}. A consequence is that even for nonproper~$f$,  $f^!$~still~has dualizing properties
for complexes having support in such a $W$ (Corollary~\ref{supports}); and there results, for $d=\sup\{\,\ell\mid H^\ell f^!\OY\ne0\,\}$ and $\omega^{}_{\!f}$  a relative dualizing sheaf, a ``generalized residue map"\looseness=-1 
\[
\int_{\>W}\<\colon H^d\Rf \R\vG^{}_{\!W}(\omega^{}_{\!f})\to\OY.
\]

Proposition~\ref{pullback1} says that
\emph{for\/ $\sE$-maps \mbox{$W\xto{\,\lift.8,g,\,}X\xto{\,\lift1.2,f,\,} Y$}\va{.75} of noetherian schemes such that\/
$fg$ is proper, and any\/ $F\in\Dqc(X),$  $G\in\Dqcpl(Y),$ the functors\/ $\LL g^*\RHqc{\sX}(F,-)$ and
\/$g^\times\RHqc{\sX}(F,-)$ both take the map\/ 
$\psi(f)G\colon f^\times G\to f^!G$ to an isomorphism.}

Section 3 gives some concrete realizations of $\psi$. Besides the examples
mentioned\- above, one has that if $R$ is a noetherian ring, $S$ a flat essentially-finite\kf-type $R$-algebra, $f\colon\spec S\to\spec R$ the corresponding scheme\kf-map, and $M$ an $R$-module, with sheafification $\sM$, then with $\env{S}\set S\otimes_R S$, the map $\psi(f)(\sM)\colon f^\times\sM\to f^!\sM$ is the sheafification of a simple $\D(S)$-map
\begin{equation}\label{affmap}
\R\<\Hom_R(S,M)\to S\Otimes{\env S}\R\<\Hom_R(S,S\otimes_R M),
\end{equation}
described in Proposition~\ref{R1.1.3.5}. So if $S\to T$ is an $R$-algebra map 
with $T$ module\kf-finite over~$R$, then, as above,  the functors
$T\Otimes{S}-$ and $\R\<\Hom_S(T, -)$ take \eqref{affmap}  to an isomorphism.

In the case where $R$ is a field, more information about the map~ \eqref{affmap} appears in Proposition~\ref{affine/k}\kf: 
the map is represented by a split $S$-module surjection with an enormous kernel.

In \S4, there are two applications of the map $\psi$. The first is to a ``reduction theorem" for the Hochschild homology of flat $\sE$-maps that was stated in \cite[Theorem 4.6]{AILN} in algebraic terms (see \eqref{AILN4.6} below), with only an indication of proof. The scheme\kf-theoretic version appears here in ~\ref{global4.6}. 

The paper \cite{AILN} also treats the nonflat algebraic case, where
$\env S$ becomes a derived tensor product. In fact, we conjecture that
the natural home of the reduction theorems is in a more general 
derived-algebraic-geometry setting.

The special case \ref{affreldual} of \eqref{AILN4.6} gives a canonical description of the relative dualizing sheaf $f^!\OY$ of a flat $\sE$-map $f\colon X\to Y$ between affine schemes. The proof is based on the known theory of $f^!$, which is constructed using arbitrary choices, such
as a compactification of $f$ or a factorization of $f$ as smooth$\smallcirc$finite; but 
the choice\kf-free formula \ref{affreldual} might be a jumping-off point for a choice\kf-free redevelopment of the underlying theory.

The second application is to a simple formula for the \emph{fundamental  class} of a flat map $f$ of affine schemes.
The fundamental class of a flat $\sE$-map \mbox{$g\colon X\to Y$}---a globalization of the Grothendieck residue map---goes from the Hochschild complex of $g$ to the relative dualizing complex~$g^!\OY$. This map is
defined in terms of sophisticated abstract notions from duality theory (see \eqref{fclass}). 
But for maps $f\colon\spec S\to \spec R$ as above, Theorem~\ref{explicit fc} says that,
with $\mu\colon S\to  \Hom_R(S,S)$  the $\env S$-homomorphism  taking\/ $s\in S$ to multiplication by\/ $s$, \emph{the fundamental class\/ is isomorphic to the sheafification of the natural composite map}
\[
S\Otimes{\env S}S\xto{\<\id\<\otimes\>\mu\>} S\Otimes{\env S}\Hom_R(S,S)\lto 
S\Otimes{\env S}\R\<\<\Hom_R(S,S).
\]
 
\medskip

\section{Preliminaries: twisted inverse image functors, essentially finite-type compactification, conjugate maps}
\label{prelims}

\begin{cosa}\label{^!} 
For a scheme $X\<$, $\D(X)$ is the derived category of $\OX$-modules, and 
$\Dqc(X)$ ($\Dqcpl(X)$) is the full subcategory spanned by the complexes with quasi-coherent cohomology modules (vanishing in all but finitely many negative degrees). We will use freely some standard functorial maps, for instance the 
projection isomorphism associated to a map $f\colon X\to Y$ of noetherian schemes (see, e.g., \cite[3.9.4]{li}):
\[\R\>\fst E\Otimes{Y}\< F\iso\R\>\fst(E\Otimes{X} \<\LL f^*\<\<F)
\quad\ \big(E\in\Dqc(X), \,F\in\Dqc(Y)\big).
\] 

Denote by $\sE$  the category of \emph{separated essentially-finite\kf-type maps of noetherian schemes.}
By \cite[5.2 and 5.3]{Nk}, there is a contravariant \mbox{$\Dqcpl$-valued} pseudo\-functor $(-)^!$ over $\sE$, determined up to isomorphism by the properties:\va2

{\rm(i)} The pseudofunctor $(-)^!$ restricts over the subcategory of proper maps in $\sE$ to a right adjoint of the derived direct-image pseudofunctor.\va1

{\rm(ii)} The pseudofunctor $(-)^!$ restricts over the subcategory of formally \'etale maps in $\sE$ to the usual inverse\kf-image pseudofunctor~$(-)^*$.\va1

{\rm(iii)} For any fiber square in\/ $\sE\!:$\va2
$$
\CD
\bullet @>v>> \bullet  \\
@V g VV @VV f V \\
\bullet@>{\vbox to 0pt{\vss\hbox{$\Xi$}\vskip.21in}} > \lift1.2,u, >\bullet
\endCD
$$
with $f,g$ proper\ and $ u,v$ formally \'etale, the base\kf-change map $\beta_{\>\Xi}$,  defined to be the
adjoint of the natural composition
\begin{equation}
\label{bcadj}
\R g_*v^*\!f^!\iso u^*\Rf f^! \longrightarrow u^*\<,
\end{equation}
is equal to the natural composite isomorphism
\begin{equation}\label{beta}
v^*\<\<f^!= v^!\<\<f^!
\iso (fv)^!=(ug)^!\iso
g^!u^!  =g^!u^*.
\end{equation}

\goodbreak

There is in fact a family of base-change \emph{isomorphisms} 
\begin{equation}\label{bch}
\beta_{\>\Xi}\colon v^*\<\<f^!\iso g^!u^*,
\end{equation}
indexed by \emph{all} commutative $\sE$-squares 
$$
\CD
\bullet @>v>> \bullet  \\
@V g VV @VV f V \\
\bullet@>{\vbox to 0pt{\vss\hbox{$\Xi$}\vskip.21in}} > \lift1.2,u, >\bullet
\endCD
$$
 that are such that in the associated diagram (which exists in $\sE$, see  \cite[\S2.2]{Nk})\looseness=-1
\[
\CD
\bullet @>i>>\bullet @>w>> \bullet  \\
@.@V h VV @VV f V \\
@. \bullet@>{\vbox to 0pt{\vss\hbox{$\Xi'$}\vskip.21in}} > \lift1.2,u, >\bullet
\endCD
\]
it holds that $\Xi'$ is a fiber square, $wi=v$ and $hi=g$, the map $u$ is flat and $i$ is formally \'etale, a family that is the unique such one that behaves transitively with~ respect to vertical and horizontal composition of such $\Xi$ (cf.~ \cite[(4.8.2)(3)]{li}), and 
satisfies:\va3

{\rm(iv)} if $\Xi$ is a fiber square with $f$ proper then the map $\beta_{\>\Xi}$ is
adjoint to the composite map \eqref{bcadj};\va2

{\rm(v)} if $f$---hence $g$---is formally \'etale, so that $f^!=f^*$ and
$g^!=g^*\<,$ then
$\beta_{\>\Xi}$~is the natural isomorphism $v^*\<\<f^*\iso g^*u^*$; and\va2

{\rm(vi)} if $u$---hence $v$---is formally \'etale, so that $u^*=u^!$ and
$v^*=v^!\<,$ then
$\beta_{\>\Xi}$~is the natural isomorphism \eqref{beta}.\va3

\noindent (For further explanation see \cite[Thm.\,4.8.3]{li} and \cite[\S5.2]{Nk}.)\va4

\noindent\emph{Remark.} With regard to (vi),  if $\Xi$ is \emph{any} commutative $\sE$-diagram with
$u$ and $v$ formally \'etale, then in the associated diagram
$i$ is necessarily formally \'etale (\cite[(17.1.3(iii) and 17.1.4]{EGA4}), so that $\beta_{\>\Xi}$ exists
\va1 (and can be identified with the canonical isomorphism $v^!\<\<f^!\iso g^!u^!$).
\end{cosa}

\begin{cosa}\label{^times}

For \emph{any} $\sE$-map $f\colon X\to Y\<$, there exists a functor 
$f^\times\colon \D(Y)\to\Dqc(X)$ that is bounded below and right-adjoint to 
$\Rf\>$. There results a $\Dqc$\kf-valued pseudofunctor $(-)^\times$ on $\sE$, for which the said adjunction is pseudo\-functorial
 \cite[Corollary (4.1.2)]{li}. Obviously, the restriction 
of~$(-)^\times$ to $\Dqcpl$ over proper maps in $\sE$
is isomorphic to that of~$(-)^!$. Accordingly, we will identify these two restricted pseudofunctors.

\end{cosa}

\begin{cosa}
\label{locimm}
Nayak's construction of $(-)^!$ is based on his extension \cite[p.\,536, Thm.\,4.1]{Nk}
of Nagata's compactification theorem, to wit, that any map~$f$ in~$\sE$ factors as
$pu$ where $p$ is proper  and $u$ is a localizing immersion (see below).  Such a factorization is called a \emph{compactification of} $f$.\va2

A \emph{localizing immersion} is an $\sE$-map $u\colon X\to Y$  for which 
every~\mbox{$y\in u(X)$} has a neighborhood 
$V=\spec A$ such that $u^{-1}V=\spec A_M$ for some multiplicatively closed subset $M\subseteq A$, see \cite[p.\,532, 2.8.8]{Nk}. For example, \emph{finite\kf-type} localizing immersions are\va1 just open immersions \cite[p.\,531, 2.8.3]{Nk}.

Any localizing immersion $u$ is formally \'etale, so that $u^!=u^*\<$.\va1
\end{cosa}

\begin{cosa}
\label{cosa:localizing-immersion}
Any localizing immersion $u\colon X\to Y$ is a flat monomorphism, whence 
\emph{the natural map\/ 
$\epsilon^{}_1\colon u^*\R u_*\iso\id_\sX\>$ is an isomorphism}: associated to the fiber~square\looseness=-1
\[
\CD
X\times_Y X @>p^{}_1>> X\\
@Vp^{}_2VV @VVuV\\
X @>>u > Y
\endCD
\]
there is the flat base\kf-change isomorphism $u^*\R u_*\iso \R p^{}_{2\>*}p_1^*$, 
and since $u$ is a monomorphism, $p^{}_1$ and $p^{}_2$ are equal isomorphisms, so that 
$\R p_{2\>*}p_1^*=\id_\sX$.

That $\epsilon^{}_1$ is an isomorphism means that the natural map is an isomorphism
\[
\Hom_{\D(X)}(E,F)\iso\Hom_{\D(Y)}(\R u_*E, \R u_*F)\qquad(E,F\in\D(X)),
\]
which implies  that the natural map $\eta^{}_{\>2}\colon\id_\sX\to u^\times \R u_*$ is  an isomorphism. 

\smallskip
\begin{small} Conversely,  \emph{any flat monomorphism\/ $f$ in\/ $\sE$ is a localizing immersion,} 
which can be  seen as follows. Using \cite[2.7\kf]{Nk} and \cite[8.11.5.1 and 17.6.1]{EGA4} one reduces to where $f$ is a  map of affine schemes, corresponding to a composite ring map $A\to B\to B_M$ with $A\to B$ \'etale and $M$ a multiplicative submonoid of~$B$. The kernel of multiplication 
 $B\otimes_A B\to B$ is generated by an idempotent $e$,  and 
 $B_M\otimes_A B_M\to B_M$ is an isomorphism, so $e$ is annihilated by an element of the form $m\otimes m\ (m\in M)$. Consequently, $B[1/m]\otimes B[1/m]\to B[1/m]$ is an isomorphism, and so replacing~$B_M$ by $B[1/m]$ reduces the problem further to the case where 
$A\to B_M$ is a finite\kf-type algebra. Finally, localizing $A$ with respect to its  submonoid of elements that are sent to units in $B_M$, one may assume further that $f$ is surjective, in which 
case \cite[17.9.1]{EGA4} gives that $f$ is an isomorphism.
\end{small}
\end{cosa}

\begin{cosa}\label{RHom^qc} 
For a noetherian scheme $X\<$, the functor
$\id_{\sX}^\times$ specified in~\S\ref{^times} is right-adjoint to the inclusion
$\Dqc(X)\hookrightarrow\D(X)$. It is sometimes called the \emph{derived quasi-coherator.} 

For any $C\in\Dqc(X)$, the unit map is an isomorphism $C\iso \id_X^\times\< C$.

For any complexes $A$ and $B$ in $\Dqc(X)$, set
\begin{equation}\label{RHqc}
\RHqc{\sX}(A,B)\set \id_X^\times\R\sHom_\sX(A,B)\in\Dqc(X).
\end{equation}

Then for $A$  in $\Dqc(X)$, 
the functor $\RHqc{\sX}(A,-)$ is right-adjoint to the endofunctor 
$-\Otimes{\sX} A$ of $\>\Dqc(X)$. 
Thus, $\Dqc(X)$ is a closed category with multiplication given by $\Otimes{\sX}$ and internal hom given by 
$\RHqc{\sX}$.

As above, the canonical $\D(X)$-map $\RHqc{\sX}(A,B)\to\R\sHom_\sX(A,B)$
is an \emph{isomorphism} whenever $\R\sHom_\sX(A,B)\in\Dqc(X)$---for example, whenever $B\in\Dqcpl(X)$ and the cohomology sheaves $H^iA$ are coherent for all $i$, vanishing for $i\gg 0$ \cite[p.\,92, 3.3]{RD}.
\end{cosa}

\begin{cosa}\label{conjugate} For categories $P$ and $Q$, let Fun$(P,Q)$ be the category of functors from $P$ to $Q$, and let Fun$^\textup{L}(P,Q)$ 
(resp.~Fun$^\textup{R}(P,Q))$ be the full subcategory spanned by the objects  that have right (resp.~left) adjoints. There is a contravariant isomorphism of categories
\[
\xi\colon\textup{Fun}^\textup{L}(P,Q)\iso\textup{Fun}^\textup{R}(P,Q)
\]
that takes any map of functors to the \emph{right-conjugate} map between the respective
right adjoints (see e.g., \cite[3.3.5--3.3.7\kf]{li}). The image under $\xi^{-1}$ of a map
of functors is its \emph{left-conjugate} map. The functor $\xi$ (resp.~$\xi^{-1}$)  takes isomorphisms of functors to isomorphisms.

For instance, for any $\sE$-map $f\colon X\to Z$ there is a bifunctorial \emph {sheafified duality}
isomorphism, with $E\in\Dqc(X)$ and $ F\in\Dqc(Z)$:
\begin{equation}\label{sheaf dual}
\R\>\fst\RHqc{\sX}(E, f^\times\<\< F\>)\iso \RHqc{Z}(\R\>\fst E,F\>),
\end{equation}
right-conjugate, for each fixed~$E$,  to the projection isomorphism 
\[
\mkern-24mu \R\>\fst(\LL f^*G\Otimes{\sX} E)\osi G\Otimes{\<\<Z}\R\>\fst E.
\]

Likewise, there is a functorial isomorphism 
\begin{equation}\label{^timesHom}
\RHqc{W}(\LL f^*G, \>f^\times\<\<H)\iso f^\times \RHqc{\sX}(G, H)
\end{equation}
right-conjugate  to the  projection isomorphism 
$\R \fst(E\Otimes{\sX} \<\LL f^*G)\osi\R\>\fst E\Otimes{\<\<Z} G.$

\end{cosa}
\section{The basic map}\label{basic map}
In this section we construct a pseudofunctorial map $\psi\colon(-)^\times\to (-)^!$. The construction is based on the following  ``fake unit" map.

\pagebreak

\begin{prop} 
\label{fake unit}
Over\/ $\sE$ there is a unique pseudofunctorial map 
\[
\eta\colon\id \to (-)^!\smallcirc\R(-)_*
\] 
whose restriction to the subcategory of proper maps in\/ $\sE$ is the unit of the adjunction between\/ $\R (-)_*$ and\/ $(-)^!,$ and such that if\/ $u$ is a localizing immersion then\/ $\eta(u)$ is inverse to the isomorphism\/ $u^!\>\R u_*= u^*\R u_*\iso\id$ in~\ref{cosa:localizing-immersion}.\looseness=-1
\end{prop}

The proof uses the next result---in which the occurrence of~$\beta_{\>\Xi}$ is justified by the remark at the end of \S\ref{^!}.  As we are dealing only with functors between derived categories, we will reduce clutter by writing $h_*$ for $\R h_*$
($h$~any map in $\sE$).

\begin{sublem} 
\label{compat}
Let\/ $\Xi$ be a commutative square  in $\sE\colon$
$$
\CD
\bullet @>v>> \bullet  \\
@V g VV @VV f V \\
\bullet@>{\vbox to 0pt{\vss\hbox{$\Xi$}\vskip.21in}} > \lift1.2,u, >\bullet
\endCD
$$
with $f,g$ proper\ and $ u,v$ localizing immersions. Let\/ $\phi^{}_{\>\Xi}\colon v_*g^!\to f^!u_*$ be the
functorial map adjoint to the natural composite map\/
\mbox{$\fst v_*g^!\<\<\iso\<\< u_*g_*g^!\to u_*.$}
Then the following natural diagram commutes.

\[
\def\1{$\id$}
\def\2{$v^*\<v_*$}
\def\3{$g^!\<g_*$}
\def\4{$v^*\<v_*g^!\<g_*$}
\def\5{$v^*\<\<f^!\<\fst v_*$}
\def\6{$v^*\<\<f^!\<u_*g_*$}
\def\7{$g^!u^*\<u_*g_*$}
\def\8{$v^!\<f^!\<\fst v_*$}
\def\9{$(fv)^!(fv)_*$}
\def\ten{$(ug)^!(ug)_*$}
\def\lvn{$g^!u^!\<u_*g_*$}
 \bpic[xscale=5, yscale=1.45]

   \node(11) at (1,-1){\2};
   \node(12) at (2,-1){\1};   
   \node(13) at (3,-1){\3};
    
   \node(22) at (2,-2){\4};
   
   \node(32) at (2,-3){\6};
   
   \node(41) at (1,-3){\5};
   \node(43) at (3,-3){\7};

   \node(51) at (1,-4){\8};
   \node(52) at (1.66,-4){\9};
   \node(53) at (2.33,-4){\ten};
   \node(54) at (3,-4){\lvn};
   
  %rows
   \draw[<-] (11)--(12) node[above=1pt, midway, scale=.75]{$\Iso$};
   \draw[->] (12)--(13) node[above=1pt, midway, scale=.75]{\textup{unit}};

  \draw[->] (41)--(32) node[above=1pt, midway, scale=.75]{$\Iso$};
  \draw[->] (32)--(43)  node[above=1pt, midway, scale=.75]{$\Iso$}
                                  node[below=1pt, midway, scale=.75]{$\beta_{\>\Xi}$};
     
   \draw[->] (51)--(52) node[above, midway, scale=.75]{$\Iso$};
   \draw[double distance=2pt] (52)--(53);
   \draw[->] (53)--(54) node[above, midway, scale=.75]{$\Iso$};

 %columns
   \draw[->] (11)--(41) node[left=1pt, midway, scale=.75]{$v^*\textup{unit}$};  
 
    \draw[->] (22)--(32) node[left=1pt, midway, scale=.75]{$v^*\phi^{}_{\>\Xi}$};  
 
   \draw[double distance=2pt] (41)--(51) ;
   \draw[->] (13)--(43) node[right=1pt, midway, scale=.75]{$\simeq$}; 
   \draw[double distance=2pt] (43)--(54) ;
   
 %slanted
  \draw[->] (11)--(22)  node[above=-6pt, midway, scale=.75]{\rotatebox{-16.5}{$\mkern43mu v^*v_*\textup{unit}$}};
  \draw[<-] (22)--(13)  node[above=-1pt, midway, scale=.75] {$\rotatebox{16.5}\Iso$};
 %labels
   \node at (2,-1.5)[scale=.9]{\circled1};
   \node at (1.5,-2.35)[scale=.9]{\circled2};
   \node at (2.5,-2.35)[scale=.9]{\circled3};   
   \node at (2,-3.56)[scale=.9]{\circled4};                                                         
                                                      
  \epic
\]

\end{sublem}

\begin{proof} Commutativity of subdiagram \circled1 is clear.

For commutativity of subdiagram \circled2, drop  $v^*$ and note the obvious commutativity of  the following adjoint of the resulting diagram:
\[
\def\1{$\fst v_*$}
\def\2{$\fst v_*g^!\<g_*$}
\def\3{$u_*g_*g^!g_*$}
\def\4{$u_*g_*$}
 \bpic[xscale=5, yscale=1.25]

   \node(11) at (1,-1){\1};
   \node(12) at (2,-1){\2};   
 
   \node(22) at (2,-2){\3};   
    
   \node(31) at (1,-3){\1};
   \node(32) at (2,-3){\4};
   
  %rows
   \draw[->] (11)--(12) ;
   \draw[->] (31)--(32) ;

 %columns 
   \draw[double distance=2pt] (11)--(31) ;
  
   \draw[->] (12)--(22) node[right=1pt, midway, scale=.75]{$\simeq$}; 
   \draw[->] (22)--(32) ;    
   
  \epic
\]

Showing commutativity of subdiagram \circled3 is similar to working out
\cite[Exercise 3.10.4(b)]{li}. (Details are left to the reader.)

Commutativity of subdiagram \circled4 is given by (vi) in section~\ref{^!}.
\end{proof}

\begin{proof}[Proof of Proposition~\ref{fake unit}]
As before, for any map~$h$ in $\sE$ we abbreviate $\R h_*$ to  $h_*$. Let $f$ be a map in $\sE$, and $f=pu$  a compactification. If $\eta$ exists, then  $\eta(f)\colon \id \to f^!\<f^{}_{\<\<*}$ must be given by the natural composition
\begin{equation}
\label{pseudounit}
\id \iso u^* u_* \xto{\via\textup{ unit}\,} u^*p^!\>p_* u_* \iso f^!\<f^{}_{\<\<*},
\end{equation} 
so that uniqueness holds.  

Let us show now that this composite map does not depend on the choice of compactification.

A morphism $r\colon(f=qv)\to(f=pu)$ from one compactification of $f$ to another is a commutative diagram of scheme\kf-maps
\begin{equation}\label{dominate}
\CD
 \bpic[xscale=2.5, yscale=1]

   \node(12) at (2,-1){$\bullet$};
      
   \node(21) at (1,-2){$\bullet$}; 
   \node(23) at (3,-2){$\bullet$};   
    
   \node(32) at (2,-3){$\bullet$};

   \draw[->] (21)--(12) node[above=1pt, midway, scale=.75]{$v$}; 
   \draw[->] (12)--(23) node[above=1pt, midway, scale=.75]{$q$}; 

   \draw[->] (21)--(32) node[below=1pt, midway, scale=.75]{$u$}; 
   \draw[->] (32)--(23) node[below=1pt, midway, scale=.75]{$p$}; 

   \draw[->] (12)--(32) node[left=1pt, midway, scale=.75]{$r$};

  \epic
\endCD
\end{equation}
If such a map $r$---necessarily proper---exists, we say that the compactification $f=qv$ \emph{dominates} $f=pu$.

\emph{Any two compactifications $X\xto{u_1\,} Z_1\xto{p^{}_1\,} Y,$ $X\xto{u_2\,} Z_2\xto{p^{}_2\,} Y$ of a given $f\colon X\to Y$ are dominated by a third one}.  Indeed, let $v\colon X\to Z_1\times_Y Z_2$ be the map corresponding to the pair $(u_1,u_2)$, let $Z\subseteq  Z_1\times_Y Z_2$ be the schematic closure 
of $v$---so that $v\colon X\to Z$ has schematically dense image\kf---and let $r_i:Z\to Z_i\  (i=1,2)$ be the maps induced by the two canonical projections. Since $u=r_iv$ is a localizing immersion, therefore, by \cite[p.\,533, 3.2]{Nk}, so is~$v$. Thus  $f=(p_ir_i)v$ is a compactification,
not depending on $i$, mapped by $r_i$ to the compactification $f=p_iu_i$.

So to show that \eqref{pseudounit} gives the same result for any two compactifications of~$f$, it suffices to do so
when one of the compactifications dominates the~other.
Thus with reference to the diagram \eqref{dominate}, and keeping in mind that $u^*=u^!$ and $v^*=v^!$, one need only show that the following natural
diagram commutes.
\[
\def\1{$\id$}
\def\2{$u^!u_*$}
\def\3{$v^!r^!r_*v_*$}
\def\4{$v^!v_*$}
\def\5{$u^!p^!p_*u_*$}
\def\6{$v^!r^!p^!p_*r_*v_*$}
\def\7{$v^!q^!q_*v_*$}
\def\8{$f^!\<\<\fst$}
 \bpic[xscale=3, yscale=1.6]

   \node(12) at (2,-1){\1};   
  
   \node(21) at (1,-2){\2};
   \node(22) at (2,-2){\3};
   \node(23) at (3,-2){\4};
   
   \node(31) at (1,-3){\5};
   \node(32) at (2,-3){\6};
   \node(33) at (3,-3){\7};
   
   \node(42) at (2,-4){\8};

  %rows
   \draw[->] (21)--(22) ;
   \draw[->] (23)--(22) ;  
   
    \draw[->] (31)--(32) ;
   \draw[->] (33)--(32) ; 
   
 %columns 
   \draw[->] (21)--(31) ;
   
   \draw[->] (22)--(32) ;
   \draw[->] (32)--(42) ;
 
   \draw[->] (23)--(33) ;

 %oblique   
   \draw[->] (12)--(21) node[right=1pt, midway, scale=.75]{$ $}; 
   \draw[->] (12)--(23) ;    
 
    \draw[->] (31)--(42) node[right=1pt, midway, scale=.75]{$ $}; 
    \draw[->] (33)--(42) ;    
    
 %labels  
   \node at (2,-1.55)[scale=.9]{\circled1};
   \node at (1.5,-2.5)[scale=.9]{\circled2};
   \node at (2.5,-2.5)[scale=.9]{\circled3};   
   \node at (1.75,-3.4)[scale=.9]{\circled4}; 
   \node at (2.25,-3.4)[scale=.9]{\circled5};      
     
  \epic
\]

Commutativity of subdiagram \circled 1 is given by Lemma~\ref{compat}, with $f\set r$ and $g\set\id_\sX$.

Commutativity of  \circled2 is clear.

Commutativity of \circled3 holds because over proper maps, $(-)^!$ and $(-)_*$ are \emph{pseudofunctorially} adjoint (see \cite[Corollary (4.1.2)]{li}.

Commutativity of \circled4 and \circled5 results from the pseudofunc\-toriality of $(-)^!$ and $(-)_*$.\va1

Thus \eqref{pseudounit} is indeed independent of choice of compactification, so that $\eta(f)$ is well-defined.\va4

Finally, it must be shown that $\eta$ is pseudofunctorial,\va{-1} i.e., for any composition
$X\xto{\,f\,}Y\xto{\,g\,}Z$ in $\sE$, the next diagram commutes:
$$
\CD
\id @>\eta(gf)>>(gf)^!(gf)_* \\
@V\eta(f)VV@VV\simeq V \\
 f^!\<\<\fst @>>f^!\eta(g)> f^!g^!g_*\fst
\endCD
$$

Consider therefore a diagram
\[
 \bpic[xscale=1.5, yscale=1.36]

   \node(31) at (1,-3){$\bullet$};   
  
   \node(21) at (1,-2){$\bullet$};
   \node(22) at (2,-2){$\bullet$};

   \node(11) at (1,-1){$\bullet$};
   \node(12) at (2,-1){$\bullet$};
   \node(13) at (3,-1){$\bullet$};

  %rows
   \draw[->] (11)--(12) node[above=1pt, midway, scale=.75]{$u$};
   \draw[->] (12)--(13) node[above=1pt, midway, scale=.75]{$w$};  

   \draw[->] (21)--(22) node[above, midway, scale=.75]{$v$};
       
 %columns 
   \draw[->] (11)--(21) node[left=1pt, midway, scale=.75]{$f$};
   \draw[->] (21)--(31) node[left=1pt, midway, scale=.75]{$g$};
  
 %oblique   
   \draw[->] (12)--(21) node[right=5pt, midway, scale=.75]{$p$}; 
   \draw[->] (13)--(22) node[right=5pt, midway, scale=.75]{$r$};    
   \draw[->] (22)--(31) node[right=5pt, midway, scale=.75]{$q$}; 
 
  \epic
\]
where $pu$ is a compactification of $f$, $qv$ of $g$, and $rw$ of $vp$---so that $(qr)(wu)$ is a compactification of $gf$. The problem then is to
show commutativity of (the border of) the following natural diagram.
\[\mkern-1.5mu
\def\1{$\id$}
\def\2{$(wu)^*(wu)_*$}
\def\3{$(wu)^*(qr)^!(qr)_*(wu)_*$}
\def\4{$(gf)^!(gf)_*$}
\def\5{$u^*u_*$}
\def\6{$u^*\<p^!p_*u_*$}
\def\7{$f^!\<\<\fst$}
\def\8{$f^!v^*v_*\fst$}
\def\9{$f^!v^*q^!q_*v_*\fst$}
\def\ten{$f^!g^!g_*\fst$}
\def\lvn{$u^*w^*w_*u_*$}
\def\twv{$u^*w^*r^!r_*w_*u_*$}
\def\thn{$u^*(rw)^!(rw)_*u_*$}
\def\frn{$u^*p^!v^*v_*p_*u_*$}
\def\ffn{$u^*w^*r^!q^!q_*r_*w_*u_*$}
\def\sxn{$u^*p^!g^!g_*p_*u_*$}
\def\svn{$u^*(vp)^!(vp)_*u_*$}
\def\egn{$u^*p^!v^*\<q^!q_*v_*p_*u_*$}
 \bpic[xscale=3.75, yscale=1.5]

   \node(11) at (1,-1){\1};   
   \node(12) at (2.29,-1){\2};
   \node(13) at (3.23,-1){\3};

   \node(21) at (1,-2){\5};
   \node(22) at (2.29,-2){\lvn};
   \node(24) at (3.9,-2){\4}; 
    
   \node(31) at (1.4,-3){\thn};
   \node(32) at (2.29,-3){\twv};
   \node(33) at (3.23,-3){\ffn};
     
   \node(41) at (1.4,-4){\svn};
   \node(42) at (2.29,-4){\frn};   
   \node(43) at (3.23,-4){\egn};   
     
   \node(51) at (1,-5){\6};
   \node(52) at (2.29,-5){\6};
   \node(53) at (3.23,-5){\sxn};
   \node(54) at (3.9,-5){\ten};     

   \node(62) at (2.29,-6){\7};
   \node(63) at (3.17,-6){\8};
   \node(64) at (3.85,-6){\9};

 %rows
	  
	   \draw[->] (11)--(12) ;
	   \draw[->] (12)--(13) ;  
                  
	   \draw[->] (21)--(22) ; 

	   \draw[<-] (31)--(32) ;
	   \draw[->] (32)--(33) ; 
	   
	   \draw[->] (41)--(42) ;
	   \draw[->] (42)--(43) ;
	 
	   \draw[double distance=2pt] (51)--(52) ;	   
	   \draw[->] (52)--(53) node[below=1pt, midway, scale=.75]{$\via\eta(g)$}; 
	   \draw[->] (53)--(54) ;

	   \draw[->] (62)--(63) ;
	   \draw[->] (63)--(64) ;  	   
	 
 %columns 
	   \draw[->] (11)--(21) ;
	   \draw[->] (21)--(51) ;
	   
	   \draw[double distance=2pt] (31)--(41) ;
	   
	   \draw[->] (12)--(22) ;
	   \draw[->] (22)--(32) ;
	   \draw[->] (32)--(42) ;
           \draw[->] (42)--(52) ;
	   \draw[->] (52)--(62) ;	
	    
	   \draw[->] (13)--(33) ;
	   \draw[->] (33)--(43) ;
 	   \draw[->] (43)--(53) ;
 
           \draw[->] (24)--(54) ;
	   \draw[->] (3.9,-5.78)--(3.9, -5.22) ;
  
 %oblique   
	   \draw[->] (13)--(24) ;
	   \draw[->] (21)--(31) node[above=-3pt, midway, scale=.75]{$\mkern45mu\eta(vp)$};

 %labels  
	   \node at (1.65,-1.55)[scale=.9]{\circled1};
	   \node at (1.85,-2.55)[scale=.9]{\circled2};
	   \node at (1.85,-3.55)[scale=.9]{\circled3};   
	   \node at (1.65,-4.55)[scale=.9]{\circled4}; 
	   \node at (2.76,-2)[scale=.9]{\circled5};      
           \node at (2.76,-3.55)[scale=.9]{\circled6};
	   \node at (3.565,-3.55)[scale=.9]{\circled7};	  
	   \node at (2.76,-4.55)[scale=.9]{\circled8}; 
	   \node at (3.23,-5.55)[scale=.9]{\circled9};     
  \epic
\]

That subdiagram \circled1 commutes is shown, e.g., in \cite[\S3.6, up to (3.6.5)]{li}. (In other words, the adjunction between $(-)^*$ and $(-)_*$ is pseudofunctorial, see \emph{ibid}., (3.6.7)(d).)

Commutativity of \circled2 is  the definition of $\eta(vp)$ via the compactification~$rw$.

Commutativity of \circled3  holds by definition of the vertical arrow on its right.

Commutativity of \circled4 (omitting $u^*$ and $u_*$) is the case $(f,g,u,v)\set(r,p,v,w)$ of Lemma~\ref{compat}. 

Commutativity of \circled5 holds because of pseudofunctoriality of the adjunction between $(-)_*$ and $(-)^!$ over \emph{proper} maps (see \S\ref{^times}).

Commutativity of \circled6 is clear. 

Commutativity of \circled7 results from  pseudofunctoriality of $(-)^!$ and~$(-)_*$. 

Commutativity of \circled8 is the definition of $\eta(g)$ via the compactification~$qv$.

Commutativity of \circled9 is simple to verify.

This completes the proof of Proposition~\ref{fake unit}. 
\end{proof}

\begin{subthm}
\label{relation}
 There is a unique pseudofunctorial map $\psi\colon(-)^\times\to(-)^!$ whose restriction over the subcategory of proper maps in $\sE$ is the identity, and such that for every localizing immersion $u,$ $\psi(u)\colon u^\times\to u^!$ is the natural composition
\[
u^\times\iso u^*\R u_* u^\times\lto u^*=u^!.
\]
\end{subthm}

%\enlargethispage*{5pt}

\begin{proof} Let $f$ be an $\sE$-map, and $f=pu$ a compactification. If $\psi$ exists, then $\psi(f)\colon f^\times\to f^!$ must be given by the natural composition
\begin{equation*}\label{psidef}
f^\times \iso u^\times p^\times = u^\times p^! \to u^!p^!\iso f^!,\\[-1pt]
\tag{\ref{relation}.1}
\end{equation*}
so that uniqueness holds. 

As for existence, using \ref{fake unit} we can take $\psi(f)$ to be the natural composition
\[
f^\times\xto{\<\<\via\> \eta\,} f^!\Rf f^\times\lto f^!.
\]
This is as required when $f$ is proper or a localizing immersion, and it behaves pseudo\-functorially, because both $\eta$ and the counit map $\Rf f^\times\to \id$ do.
\end{proof}

\begin{subrem}\label{eta from psi} Conversely, one can recover $\eta$ from $\psi$: 
it is simple to show~that for any $\sE$-map $f\colon X\to Y$ and
$E\in\Dqc(X)$, and with $\eta^{}_{\>2}\colon \id_X\to f^\times\R\fst$ the unit map of the adjunction $f^\times\! \dashv \R\fst$, one has
\begin{equation*}\label{eta from psi.1}
\eta(E) = \psi(f)(\R\fst E)\smallcirc\eta^{}_{\>2}(E).
\tag{\ref{eta from psi}.1}
\end{equation*}
(\emph{Notation}: $\mathsf F\dashv\mathsf G$ signifies that the functor $\mathsf F$ is left-adjoint to the functor~$\mathsf G$.)
\end{subrem}

\begin{subrem}
If $u\colon X\to Y$ is a localizing immersion, then the  map
\[
u^\times\R u_*\xto{\!\psi(u)\>} u^*\R u_*\iso \id.
\] 
is an \emph{isomorphism,} inverse to the isomorphism $\eta^{}_{\>2}$  in 
\S\ref{cosa:localizing-immersion}. (A proof is left to the reader.)
\end{subrem}

\begin{subrem}\label{L1.1.25}
\noindent   \emph{The map\/ \va{-.6} $\R\>u_*\psi(u)\colon\R\>u_*u^\times\to\R\>u_*u^*$
is equal to the composite} 
\[
\R\>u_*u^\times\xto{\epsilon_2^{}\>\>}\id\xto{\eta_1^{}\>}\R\>u_*u^*,
\]
 where $\epsilon_2^{}$ is the counit of the adjunction $\R\>u_*\<\dashv u^\times$ and $\eta_1^{}$ is the unit of the adjunction $u^*\<\dashv\R\>u_*$.

Indeed, by \S\ref{cosa:localizing-immersion} the counit $\epsilon^{}_1$ of the adjunction $u^*\<\dashv\R\>u_*$ is an isomorphism;   and since the composite
\[
\R\>u_*\xto{\<\eta_1^{}\R\>u_*\>}\R\>u_*u^*\R\>u_*\xto{\<\R\>u_*\epsilon_1^{}\>\>}\R\>u_*
\]
is the identity map, therefore  $\R\>u_*\epsilon_1^{-\<1}=\eta_1^{}\R\>u_*$, as both are the (unique) inverse of~$\R\>u_*\epsilon_1^{}$; so the next diagram commutes, giving the assertion:
\[
\CD
\R\>u_*u^\times @>\R\>u_*\epsilon_1^{-\<1}u^\times=\:\eta_1^{}\R\>u_*u^\times\,>> \R\>u_*u^*\R\>u_*u^\times \\
@V{\epsilon_2^{}}VV @VV{ \R\>u_*u^*\epsilon_2^{}}V \\
\id @>>{\qquad\quad\eta_1^{}\qquad\qquad}> \R\>u_*u^*
\endCD
\]
\vskip2pt

Also, using the  isomorphism $\epsilon^{}_1\colon u^*\R u_*\iso \id$ 
(resp.,  its right conjugate~$\eta^{}_{\>2}\colon \id\iso u^\times \R u_*$), one can
recover~$\psi(u)$ from~$\R u_*\psi(u)$ by applying the functor $u^*$ 
(resp.~$u^\times$), thereby obtaining \mbox{alternate} definitions of $\psi(u)$.
\end{subrem}

\smallskip
The next Proposition asserts compatibility of $\psi$ with the flat base-change maps for $(-)^!$ (see \eqref{bch}) and for $(-)^\times\<.$

\begin{prop}\label{base change} 
Let\/ $f\colon X\to Z$ and\/ $g\colon Y\to Z$ be maps in\/ $\sE,$ with $g$ flat. 
Let\/ $p\colon X\times_Z Y\to X$ and $q\colon X\times_Z Y\to Y$ be the projections. Let $\beta\colon p^*\<\<f^\times\to q^\times\< g^*$ be the map adjoint to the natural composite map\/ 
\[\R q_*\>p^*\<\<f^\times\iso g^*\R\fst f^\times\to g^*\<.
\]
Then the following diagram commutes.
\[
\CD
p^*\<\<f^\times @>p^*\psi(f)>> p^*\<\<f^!\\
@V\beta VV @VV\eqref{bch}V \\
q^\times\< g^* @>>\psi(q)> q^!g^*
\endCD
\]
\end{prop}

\begin{proof} 
Let $f=\bar f u$ be a compactification, so that there is a composite cartesian diagram (with $h$ flat 
and with $\bar q v$ a compactification of $q$):
\[
\CD
X\times_ZY @>p>> X\\
@VvVV @VVuV\\
W\times_ZY @>h>> W\\
@V\bar q VV @VV\bar f V\\
Y @>>g> Z
\endCD
\]
In view of the pseudofunctoriality of $\psi$, what needs to be shown is commutativity of the following natural diagram.
\[
\def\1{$p^*u^{\<\times}\!\bar f^\times$}
\def\2{$p^*u^*\<\<\bar f^{\>\>!}$}
\def\3{$p^*\<\<f^!$}
\def\4{$v^\times\< h^*\<\<\bar f^{\times}$}
\def\5{$v^* \<h^*\<\<\bar f^{\>\>!}$}
\def\6{$v^\times \<\bar q^{\times}\<\<g^*$}
\def\7{$v^* \bar q^{\>\>!}\<\<g^*$}
\def\8{$q^!\<g^*$}
\def\9{$p^*\<\<f^{\<\times}$}
\def\0{$q^\times\< g^*$}
 \bpic[xscale=3, yscale=1.5]

   \node(10) at (0,-1){\9} ; 
   \node(11) at (1,-1){\1};   
   \node(12) at (2,-1){\2};
   \node(13) at (3,-1){\3};

   \node(21) at (1,-2){\4};
   \node(22) at (2,-2){\5};

   \node(30) at (0,-3){\0} ;    
   \node(31) at (1,-3){\6};
   \node(32) at (2,-3){\7};
   \node(33) at (3,-3){\8};

 %rows
	   \draw[->] (10)--(11) node[above=1pt, midway, scale=.75]{$\Iso$};
	   \draw[->] (11)--(12) ;
	   \draw[->] (12)--(13) node[above=1pt, midway, scale=.75]{$\Iso$};  
                  
	   \draw[->] (21)--(22) ; 
  
           \draw[->] (30)--(31) node[above=1pt, midway, scale=.75]{$\Iso$};
	   \draw[->] (31)--(32) ;
	   \draw[->] (32)--(33) node[above=1pt, midway, scale=.75]{$\Iso$};

 %columns 
 	   \draw[->] (10)--(30) ;

	   \draw[->] (11)--(21) ;
	   \draw[->] (21)--(31) ;
	   	   
	   \draw[->] (12)--(22) ;
	   \draw[->] (22)--(32) ;
	 	    
	   \draw[->] (13)--(33) ;

 %labels  
	   \node at (1.5,-1.51)[scale=.9]{\circled1};
	   \node at (1.5,-2.51)[scale=.9]{\circled2};
  \epic
\]

Commutativity of each of the unlabeled subdiagrams is an instance of transitivity of the 
appropriate base\kf-change map (see e.g., \cite[Thm.\,(4.8.3)]{li}).

Commutativity of \circled2 is straightforward to verify.

Subdiagram \circled1, without $\bar f^{\>\sss\bullet}$, expands naturally as follows (where
we have written $u_*$ (resp.~$v_*$) for $\R u_*$ (resp.~$\R v_*$)):
\[
\def\1{$p^*\<u^{\<\times}$}
\def\2{$p^*\<u^*u_*u^{\<\times}$}
\def\3{$p^*u^*$}
\def\4{$v^*v_* \>p^*\<u^\times$}
\def\5{$v^*\<h^*\<u_*u^\times$}
\def\6{$v^\times\< h^*$}
\def\7{$v^*v_*v^\times \<h^*$}
\def\8{$v^*\<h^*$}
 \bpic[xscale=3.3, yscale=1.5]

   \node(10) at (0,-1){\1} ; 
   \node(12) at (2,-1){\2};
   \node(13) at (3,-1){\3};
  
   \node(21) at (1,-2){\4};
   \node(22) at (2,-2){\5};

   \node(30) at (0,-3){\6};
   \node(32) at (2,-3){\7};
   \node(33) at (3,-3){\8};

 %rows
	   \draw[<-] (10)--(12) node[above=1pt, midway, scale=.75]{$$};
	   \draw[->] (12)--(13) node[above=1pt, midway, scale=.75]{$$};  
                  
	   \draw[<-] (21)--(22) ; 
  
           \draw[<-] (30)--(32) node[above=1pt, midway, scale=.75]{$$};
	   \draw[->] (32)--(33) node[above=1pt, midway, scale=.75]{$$}; 
	   
 %columns 
 	   \draw[->] (10)--(30) ;

	   \draw[->] (12)--(22) ;
	   \draw[->] (13)--(33) ;

 %oblique	   	   
	   \draw[<-] (10)--(21) ;
	   \draw[->] (21)--(32) ;
	 	    
	   \draw[->] (22)--(33) ;

 %labels  
	   \node at (1.3,-1.51)[scale=.9]{\circled3};
	   \node at (2,-2.51)[scale=.9]{\circled4};
  \epic
\]

Here the unlabeled diagrams clearly commute.

Commutativity of \circled3 results from the fact that the natural isomorphism $h^*u_*\to v_*p^*$ is adjoint to
the natural composition \mbox{$v^*h^*u_*\iso p^*u^*u_*\to p^*$} (see \cite[3.7.2(c)]{li}).

Commutativity of \circled4 results from the fact that the base\kf-change map 
$p^*u^\times\to v^\times h^*$ is adjoint to  $v_*p^*\iso h^*u_*u^\times\to h^*\<$.

Thus \circled1 commutes; and Proposition~\ref{base change} is proved. 
\end{proof}

\begin{cosa}
\label{proper-support}
Next we treat the interaction of the map $\psi$ with standard derived functors. Our approach involves the notion of \emph{support,} reviewed in Appendix~\ref{Support}. 

\begin{sublem}
\label{L1.1}
Let\/  $u\colon X\to Z$ be a localizing immersion, 
\mbox{$\epsilon_2^{}\colon \R\>u_*u^\times\to\id$}  the counit of the adjunction\/ 
$\R\>u_*\<\dashv u^\times,$ and\/ $\eta^{}_1\colon \id\to\R\>u_*u^*$ the unit\va{.6} of the adjunction\/ $u^*\dashv\R\>u_*$.  
For all $E\in\Dqc(X)$  and\/ $F\in\Dqc(Z),$ the maps\/ $\R u_*E\Otimes{\<Z}\eta^{}_1( F)$ and\/~$\RHqc{Z}\big(\R u_* E,\epsilon^{}_2( F)\big)$ are isomorphisms.
\end{sublem}

\begin{proof} Projection isomorphisms make the map 
$\R u_*E\Otimes{\<Z}\eta^{}_1$ isomorphic to\looseness=-1
\[
\R u_*(E\Otimes{\sX}u^*)
\xto{\!\via\>u^*\eta^{}_1}\R u_*(E\Otimes{\sX}u^*\R u_*u^*).
\]
Since $u^*\eta^{}_1$ is an isomorphism (with inverse the isomorphism
$u^*\R u_*u^*\iso u^*$ from \ref{cosa:localizing-immersion}), therefore 
so is $\R u_*E\Otimes{\<Z}\eta^{}_1$.\va1

Similarly, to show that $\RHqc{Z}\big(\R u_* E,\epsilon^{}_2\big)$ is an  isomorphism, one can use the duality isomorphism \eqref{sheaf dual} to reduce to noting that $u^\times\epsilon^{}_2$ is an isomorphism because it is 
right-conjugate to the inverse of the isomorphism 
$\R u_* \eta^{}_1:\R u_*u^*\R u_*\iso \R u_*$.
\end{proof}

\begin{subprop}
\label{Gampsi}
Let\/ $f\colon X\to Y$ be a map in\/ $\sE,$ $W$ a union of closed subsets of\/ $X\>$ to each of \kf which
the restriction of\/ $f\>$  is proper, and\/ \mbox{$E\in\Dqc(X)$} a complex with support\/ 
$\supp(E)$ contained in\/ $W$.\va{.6} Then the functors\/ $\R\vG^{}_{\!W}(-)$, $E\Otimes{\sX}(-)$ and\/  
$\RHqc{\sX}(E,-)$   take\/ $\psi(f)\colon f^\times\to f^!$ to an isomorphism.\looseness=-1
\end{subprop}

\begin{proof} 
By~\ref{Hom and tensor 0}(ii), it is enough to prove that Proposition~\ref{Gampsi}\va{.5} holds for
one~$E$ with \mbox{$\supp(E)=W\<$,} like\va1
 $E=\R\vG^{}_{\!W}\>\OX$ (see ~\ref{suppRg}). For such an~$E$,
\ref{Hom and tensor 0}~shows it enough to prove that $\RHqc{\sX}(E,\psi(f))$
is an isomorphism.\va1

Let $X\xto{u\>\>} Z\xto{p\>\>}Y$ be a compactification of $f$ (\S\ref{locimm}). In view of \eqref{psidef}, we need only treat the case $f=u$. In this case  
 it suffices to show, with $\epsilon^{}_2$ and $\eta^{}_1$
as in Remark~\ref{L1.1.25}, that
\begin{eqnarray*}
\RHqc{Z}\big(\R\>u_*E,\eta^{}_1\epsilon^{}_2\big) &\cong & 
\RHqc{Z}\big(\R\>u_*E,\R\>u_*\psi(u)\big)\\
&\cong& \R u_*\RHqc{\sX}\big(u^*\R\>u_*E,\psi(u)\big)\\
&\cong& \R u_*\RHqc{\sX}\big(E,\psi(u)\big)
\end{eqnarray*}
is an isomorphism.

Lemma~\ref{L1.1} gives that $\RHqc{Z}(\R\>u_*E,\epsilon^{}_2)$ is an isomorphism. It remains to be shown  that $\RHqc{Z}(\R\>u_{*}E,\eta^{}_{1})$ is an isomorphism.

The localizing immersion $u$ maps $X$ homeomorphically onto $u(X)$ \cite[2.8.2]{Nk}, so we can regard $X$ as a topological subspace of~$Z$. Let 
$i\colon V\hookrightarrow X$ be the inclusion into $X$ of a subscheme such that the restriction $fi=pui$ is proper. Then $ui$ is proper, and so $V$ is a closed subset of $Z\<$. Thus $W=\supp_\sX(E)=\supp_Z(\R\>u_*E)$ (see Remark~\ref{supp u_*}) is a union of subsets of~$X$ that are closed in $Z$. 
So Proposition~\ref{Hom and tensor 0} can be applied to show that,
since, by Lemma~\ref{L1.1}, $\R u_*E\Otimes{\<Z}\eta^{}_1$ is an isomorphism, therefore  $\RHqc{Z}(\R\>u_{*}E,\eta^{}_{1})$ is an isomorphism, as required.
\end{proof}

Let $W\subseteq X$ be as in Proposition~\ref{Gampsi}. Let $\Dqc{}(X)_{W}\subseteq\Dqc(X)$ be the essential image of 
$\R\vG^{}_{\!W}(X)$---the full subcategory spanned by the complexes that are exact outside $W\<$. By Lemma~\ref{Supp}, any $E\in\Dqc(X)_W$ satisfies
$\supp E\subseteq W$. Arguing as in \cite[\S2.3]{AJS}  one finds that the two
natural maps from $\R\vG^{}_{\!W}\R\vG^{}_{\!W}$ to $\R\vG^{}_{\!W}$ are 
\emph{equal isomorphisms;} and deduces that 
the natural map is an isomorphism 
\[
\Hom_{\D(X)}(E,\R\vG^{}_{\!W}F)
\iso\Hom_{\D(X)}(E,F)
\quad \big(E\in\Dqc(X)_{W},\,F\in\Dqc(Y)\big),
\]
with inverse the natural composition 
\[
\Hom_{\D(X)}(E,F)\to\Hom_{\D(X)}(\R\vG^{}_{\!W}E,\R\vG^{}_{\!W}F)\iso
\Hom_{\D(X)}(E,\R\vG^{}_{\!W}F).
\]

\begin{subcor}
\label{supports}
With the preceding notation, $\Rf\colon\Dqc{}(X)_{W}\to\Dqc(Y)$ has as right adjoint the functor\/ $\R\vG^{}_{\!W}f^\times\<$. 
When restricted to\/ $\Dqcpl(Y),$ this right adjoint is isomorphic to\/ 
$\R\vG^{}_{\!W}f^!\<$.
\end{subcor}

\begin{proof}
For $E\in\Dqc(X)_{W}$ and $G\in\Dqc(Y)$, there are natural isomorphisms
\begin{align*}
\Hom_{\D(Y)}(\Rf E,G)
&\cong
\Hom_{\D(X)}(E,f^\times\< G)\\
&\cong 
\Hom_{\D(X)}(E,\R\vG^{}_{\!W}f^\times\<G)
\underset{\!\textup{\ref{Gampsi}}\,}\iso 
\Hom_{\D(X)}(E,\R\vG^{}_{\!W}f^!G). 
\end{align*}
\end{proof}

\begin{subrem}\label{barint}  
The preceding Corollary entails the existence of a counit map
\[
\bar{\>{\int}_{\!\!\!W}}\colon \Rf \R\vG^{}_{\!W} f^!\OY\to \OY.
\]
Factoring $f$ over suitable affine open subsets $U$ as $U\xto{\<i_U\>}Z\xto{\<h_U\>}Y$ 
where $i_U$ is finite and 
$h_U$ is essentially smooth,\va{.5} one gets that  $i_{U*}f^!\OY|_U$ is of the form 
$\R\>\sHom_Z(\R\> i_{U*}\OX,\Omega_{h_U}^n[n])$ for some $n=n_U$\va{.6} such that the sheaf 
$\Omega_{h_U}^n$ of relative $n$-forms is free of rank 1;\va{-.5} and hence local depth considerations\va{.6}
imply that there is an integer $d$\va1  such that $H^{-e}f^!\OY=0$ for all
$e> d$, while $\omega^{}_{\!f}\set H^{-d}f^!\OY\ne 0$. This $\omega_{f}$, determined up to isomorphism by $f\<$, is a \emph{relative dualizing sheaf} (or \emph{relative canonical sheaf}\kf) of $f$.

There results a natural composite map of $\OY$-modules
\begin{align*}\label{int}
\int_{\>W}\<\colon H^d\Rf \R\vG^{}_{\!W}(\omega^{}_{\!f})
&=H^0\Rf \R\vG^{}_{\!W}(\omega^{}_{\!f}[d])\\[-5pt]
&\lto H^0(\Rf \R\vG^{}_{\!W} f^!\OY)
\xto{\via\!\!\bar{\,\,\>\int^{}_{\<W}}}
H^0\OY=\OY,
\end{align*}
that generalizes the map denoted ``$\textup{res}_Z$" in \cite[\S3.1]{Sa}.  

A deeper study of this map\va{-1} 
involves the realization  of $\omega^{}_{\!f}\>$, for certain~$f\<$, in terms of regular differential forms, and the resulting relation of~$\int^{}_{W}$  with residues of differential forms,\va{.4} 
cf.~ \cite{HK1} and \cite{HK2}. See also \S\ref{intres} below.

\end{subrem}

\end{cosa}

\begin{subprop}
\label{pullback1}
Let\/ $W\xto{\,g\,} X\xto{\,f\,} Y$ be $\sE$-maps such that\/ $f\<g$ is proper.\va1 For any\/ $F\in\Dqc(X)$ and $G\in\Dqcpl(Y),$ 
the maps
\begin{equation*}
\label{via psi}
\LL g^*\RHqc{\sX}(F,f^\times\< G) \xto{\via \psi}\LL g^*\RHqc{\sX}(F,f^! G)
\tag{\ref{pullback1}.1}
\end{equation*}
\begin{equation*}\label{via psi'}
g^\times\RHqc{\sX}(F,f^\times\< G) \xto{\via \psi}g^\times\RHqc{\sX}(F,f^! G)
\tag{\ref{pullback1}.2}
\end{equation*}
\begin{equation*}\label{via psi''}
g^\times\R\>\sHom_{\sX}(F,f^\times\< G) \xto{\via \psi}g^\times\R\>\sHom_{\sX}(F,f^! G)
\tag*{(\ref{pullback1}.2)$'$}
\end{equation*}
are isomorphisms.
\end{subprop}

\begin{proof} Since $g^\times\<\<\id_\sX^\times\cong (\id_\sX\!\smallcirc g)^\times=g^\times\<$, therefore  \eqref{via psi'} is an isomorphism if and only if so is \ref{via psi''}. (Recall that $\RHqc{\sX}=\id_\sX^\times\R\sHom_\sX$.) \va1

As for \eqref{via psi'} and \eqref{via psi}, note first that
the proper map $g$ induces a surjection $g^{}_2$ of~$\>W$ onto a closed subscheme $V$ of~$X\<$; 
 so $g=g^{}_1g^{}_2$ with $g^{}_1$ a 
closed~immersion and $g^{}_2$  surjective. \va1

Let $X\xto{u\>\>}Z\xto{p\>\>}Y$ be a compactification of $f$.
Since $p\>ug^{}_1g^{}_2$ is proper, so is $ug^{}_1g^{}_2$, whence $ug^{}_1$ maps $V=g^{}_2(g^{-\<1}_2V)$\va{.6} homeomorphically onto a closed subset of $Z$, and for each 
$x\in V$ the natural map\va{.6} $\sO_{\<Z\<, \>ug^{}_1\<x}\to\sO_{\<V\<\<,\>\>x}$ is a surjection (see \cite[2.8.2]{Nk}); thus $ug^{}_1$ is a closed immersion, and therefore 
$f\<g^{}_1=p\>ug^{}_1$ is of finite type, hence, by \cite[5.4.3]{EGA2}, proper (since $fg^{}_1g^{}_2$ is).\va1

Since $\LL g^*=\LL g_2^*\LL g_1^*$ and $g^\times=g_2^\times g_1^\times$, 
it suffices that the Proposition hold when $g=g_1^{},$ i.e., we may assume that $g\colon W\to X$ \emph{is a closed immersion.} It's enough then to 
show that \eqref{via psi} and \eqref{via psi'} become isomorphisms after 
application of the functor $g_*\>$.

Via  projection isomorphisms, the map
\[
g_*\LL g^*\RHqc{\sX}(F,f^\times\< G) \xto{\<g_{\mkern-.5mu*}\<\eqref{via psi}\>}g_*\LL g^*\RHqc{\sX}(F,f^! G)
\]
is isomorphic to the map
\begin{equation*}\label{map3}
g_*\OW\Otimes{\sX}\RHqc{\sX}(F,f^\times\< G)
\xto{\<\<\via \psi\>\>} g_*\OW\Otimes{\sX}\RHqc{\sX}(F,f^!G);
\tag{\ref{pullback1}.3}
\end{equation*}
and making the substitution
\[ 
\big(f\colon X\to Z, E, F\big)\mapsto \big(g\colon W\to X, \OW, \RHqc{\sX}(F,f^\times\< G)\big)
\] in the isomorphism \eqref{sheaf dual} leads to an isomorphism between the map
\[
g_*g^\times\RHqc{\sX}(F,f^\times\< G) \xto{\<g_{\mkern-.5mu*}\<\eqref{via psi'}\>}g_*g^\times\RHqc{\sX}(F,f^!G)
\] 
and the map\va3
\begin{equation*}\label{map4}
\RHqc{\sX}(g_*\OW, \>\RHqc{\sX}(F,f^\times\< G)) \xto{\via \psi}
\RHqc{\sX}(g_*\OW, \>\RHqc{\sX}(F,f^! G)).
\tag{\ref{pullback1}.4}
\end{equation*}
Via adjunction and projection isomorphisms, \eqref{map4} is isomorphic to \begin{equation*}\label{map5}
\RHqc{\sX}(g_*\LL g^*\<\<F, f^\times\< G) \xto{\via \psi}
\RHqc{\sX}(g_*\LL g^*\<\<F, f^! G).
\tag{\ref{pullback1}.5}
\end{equation*}
By Lemma~\ref{Supp}, $\supp(g_*\LL g^*\<\<F)\subseteq\Supp(g_*\LL g^*\<\<F)\subseteq W\<$,
so \ref{Gampsi} gives that \eqref{map5} is an isomorphism, whence so is \eqref{map4}.

Thus \eqref{via psi'} is an isomorphism. Also, \mbox{$\supp(g_*\OW)=\Supp(\OW)=W$,} so~\ref{Hom and tensor 0} shows that  \eqref{map3} is an isomorphism, whence so is \eqref{via psi}.
\end{proof}

\pagebreak
\begin{subrem}\label{for fTd} (Added in proof.)
 For an $\sE$-map $f$, with compactification $f=pu$,  set 
\[
(p,u)^!\>G\set u^!p^!\>G\qquad\big(G\in\Dqc(Y)\big).
\]
It is shown in \cite[Section 4]{Nm4} that $(p,u)^!$
depends only on $f\<$, in the sense that up to canonical isomorphism
$(p,u)^!$ is independent of the factorization $f=pu$. (When this paper was written
it was known only that $(p,u)^!G$ is canonically isomorphic to $f^!G$
when $G\in\Dqcpl(Y)$.) Likewise, for all $G\in\Dqc(Y)$ the functorial map 
 \[
\psi(p,u)(G)\colon f^\times G \iso u^\times p^\times G= u^\times p^! G\xto{\!\psi(u)p^!} u^!p^!G= (p,u)^!G
\]
depends only on $f$ \cite[Section 8]{Nm4}. So one may set $f^!\set(p,u)^!$ and $\psi(f)\set\psi(p,u)$; and then the preceding proof of
Proposition~\ref{pullback1} works for \emph{all} $G\in\Dqc(Y)$.
\end{subrem}

\section{Examples}\label{section:examples}
Corollaries \ref {psi via RHom}--\ref{varGM} provide concrete interpretations of the map~ 
$\psi(u)$ for certain localizing immersions $u$. 

Proposition~\ref{R1.1.3.5} gives a purely algebraic expression
for $\psi(f)$ when $f$ is a \emph{flat} $\sE$-map between affine schemes.  
An elaboration for when the target of $f$ is the Spec of a field  is given in Proposition~\ref{affine/k}. The scheme\kf-theoretic results ~\ref{relation},  \ref{base change} and~\ref{pullback1}
tell us some facts about the pseudofunctorial behavior of $\psi(f)$; but how to prove these facts by purely algebraic arguments is left open.

\medbreak
\begin{lem}
\label{RHom} 
Let\/ $f\colon X\to Z$ be an\/ $\sE$-map, and let\/ $F\in\Dqc(Z)$. The  functorial isomorphism~$\zeta(F\>)$ inverse to that gotten by setting\/ $E=\OX$ in\/~ \eqref{sheaf dual}
makes the following, otherwise natural, functorial diagram commute$\>:$
\[
\def\1{$\R\>\fst f^\times\<\< F$}
\def\3{$\RHqc{Z}(\R\>\fst\OX,F\>)$}
\def\4{$\RHqc{Z}(\OZ,F\>)$}
\def\5{$F$}
 \bpic[xscale=6, yscale=1.0]

   \node(11) at (1,-1){\3};   
   \node(12) at (2,-1){\1}; 
  
   \node(31) at (1,-3){\4};  
   \node(32) at (2,-3){\5};
   
  %rows
    \draw[->] (11)--(12) node[above=1pt, midway, scale=.75] {$\zeta(F\>)$};  
    
    \draw[->] (31)--(32) node[above=1pt, midway, scale=.75] {$\Iso$}; 
    
 %columns 
     \draw[->] (11)--(31) ;
 
     \draw[->] (12)--(32) ;

    \epic
\]
\end{lem}

\begin{proof} Abbreviating $\R\>\fst$ to $\fst$ and $\LL f^*$ to $f^*\<\<$, one checks that the diagram in question is right-conjugate to the 
natural diagram, functorial in $G\in\Dqc(Z)$,
\[
\def\1{$G\Otimes{Z}\fst \OX$}
\def\2{$\fst(f^*G\Otimes{X}\OX)$}
\def\3{$\fst f^*G$}
\def\4{$G,$}
\def\5{$G\Otimes{Z}\OZ$}
 \bpic[xscale=3.7, yscale=2]

   \node(11) at (1,-1){\1};   
   \node(12) at (2,-1){\2}; 
 
   \node(13) at (3,-1){\3}; 
 
   \node(21) at (1,-2){\5};  
   \node(23) at (3,-2){\4};
   
  %rows
    \draw[<-] (11)--(12) node[above=1pt, midway, scale=.75] {$\textup{projection}$};  
    \draw[<-] (12)--(13) node[above=1pt, midway, scale=.75] {$\Iso$}; 
 
    \draw[<-] (21)--(23) node[above=1pt, midway, scale=.75] {$\Iso$}; 
    
 %columns 
     \draw[<-] (11)--(21) ;

     \draw[<-] (13)--(23) ;    
    
   \epic
\]
whose commutativity  is given by \cite[3.4.7(ii)]{li}.
\end{proof}

\begin{subcor}\label{psi via RHom} For any localizing immersion\/  $u\colon X\to Z$  and $F\in\Dqc(Z),$ 
the map $\psi(u)\big(F\big)$ from ~\textup{\ref{relation}}
is isomorphic to the natural composite map
\[
u^*\RHqc{Z}(\R u_*\OX,F\>)\lto
u^*\RHqc{Z}(\OZ,F\>)\iso u^*\<\<F.
\]
\end{subcor}

\begin{proof} 
This is immediate from \ref{RHom} (with $f=u$).
\end{proof}

For the next Corollary recall that, when $Z=\spec R$, the sheafification functor 
${}^\sim={}^{\sim_R}$ is an isomorphism
from $\D(R)$ to the derived category of quasi-coherent $\OZ$-modules, whose inclusion into
$\Dqc(Z)$ is an equivalence of categories \cite[p.\,230, 5.5]{BN}.

\begin{subcor}\label{affine locimm}
In\/ \textup{\ref{psi via RHom},} if\/ $X=\spec S$ and\/ $Z=\spec R$ are affine---so that $u$ corresponds to a flat epimorphic ring homomorphism $R\to S$---and\/ $M\in\D(R),$
then\/ $\psi(u)(M^\sim)$ is the sheafification of the natural \/ $\D(S)$-map
\[
\R\<\<\Hom_R(S,M)\cong S\otimes_R \R\<\<\Hom_R(S,M)\to S\otimes_R (\R\<\<\Hom_R(R,M))=S\otimes_R M.
\]
\end{subcor}

\begin{proof} Use the following well-known facts:\va2

1. $\RHqc{Z}(A^\sim\<,B^\sim)\cong\R\<\<\Hom^{}_R(A,B)^\sim\qquad\big(A,B\in\D(R)\big)$.\va2

\begin{small}
This results from the sequence of natural isomorphisms, for any $C\in\D(R)$:
\begin{align*}
\Hom_{\D(Z)}(C^\sim\<,\R\<\<\Hom_R(A,B)^\sim)&\cong \Hom_{\D(R)}(C,\R\<\<\Hom_R(A,B))\\
&\cong \Hom_{\D(R)}(C\Otimes{R}\!A,B)\\
&\cong \Hom_{\D(Z)}((C\Otimes{R}\!A)^\sim\<,B^\sim)\\
&\cong \Hom_{\D(Z)}(C^\sim\!\Otimes{Z}\!\<A^\sim\<,B^\sim)\\
&\cong \Hom_{\D(Z)}(C^\sim\<,\>\RHqc{Z}(A^\sim\<,B^\sim)).
\end{align*}
\end{small}
\vskip-8pt
2. $\R\<\<\Hom_R(S,M)^{\sim_R}= u_*\R\<\<\Hom_R(S,M)^{\sim_S}$.\va1

3. $u^*(A^{\sim_R})=(S\otimes_R A)^{\sim_S}\qquad\big(A\in\D(R)\big).$\va1

4. For any $N\in\D(S)$, the natural $\D(Z)$-map $u^*\R u_*N^{\sim_S}\to N^{\sim_S}$ is the sheafification of the natural $\D(S)$-map $S\otimes_R N\to N$.
\end{proof}

\begin{subcor}\label{varGM}
 Let\/ $R$ be a  noetherian ring that is complete with respect to the topology defined by an ideal~ $I\<,$ let\/ $p\colon Z\to \spec R$ be a  proper map,
and let\/ $X\set(Z\setminus p^{-1}\spec R/I)\overset {\lift .5,u\>\>\>,}{\hookrightarrow} Z$ be the inclusion.\va{.7} 
For any\/ $F\in\Dqc(Z)$ whose cohomology modules are all coherent, $u^\times\<\< F=0$.
\end{subcor}

\begin{proof} Since $u^*\R u_*u^\times F\cong u^\times F$ 
(\S\ref{cosa:localizing-immersion}), it suffices that $\R u_*u^\times F=0$,
that is,  by~\ref{psi via RHom}, that 
$\RHqc{Z}(\R u_*\OX, F\>)=0$.

Set $W\set p^{-1}\spec (R/I)$. There is a natural triangle
\[
\RHqc{Z}(\R u_*\OX, F\>)
\to \RHqc{Z}(\OZ, F\>)\xto{\alpha\>} \RHqc{Z}(\R\vG^{}_{\!W}\OZ, F\>)
\xto{+\,}
\]
It is enough therefore to show that $\alpha$ is an isomorphism.

Let $\kappa\colon Z_{\</W}\to Z$ be the formal completion of $Z$ along $W\<$. For any $\OZ$-module
$G$ let\/ $G_{\</W}$ be the completion of $G$---an $\mathcal O_{Z_{\</W}}$-module; and let $\Lambda_W$ be the functor given objectwise by $\kappa_*G_{\</W}$. The composition of $\alpha$ with the ``Greenlees\kf-May" isomorphism
\[
\RHqc{Z}(\R\vG^{}_{\!W}\OZ, F\>)\iso \id_Z^\times\LL\Lambda_W F,
\]
given by \cite[0.3]{AJL1} is, by \emph{loc.\,cit.}, the map~
$\id_Z^\times\<\<\lambda$, where 
$\lambda\colon F\to\LL\Lambda_W F$ is the unique map 
whose composition with the canonical map
$\LL\Lambda_W F\to\Lambda_W F$ is the completion map $F\to\Lambda_W F$.
So we need $\id_Z^\times\<\<\lambda$ to be an isomorphism.
Hence, the isomorphisms $F=\id_Z^\times\<\<F\iso \id_Z^\times\< \kappa_*\kappa^*\<\<F$ in \cite[3.3.1(2)]{AJL2} (where $\id_Z^\times$
is denoted $\R Q$) and  
$
\lambda^*_*\colon\kappa_*\kappa^*\<\<F\iso\LL\Lambda_W F
$ 
in \cite[0.4.1]{AJL1} (which requires coherence of the cohomology of $F$)  
reduce the problem to showing that \emph{the natural composite map}\looseness=-1
$$
F\to\kappa_*\kappa^*\<\<F\xto{\<\lambda^*_*\>}\LL\Lambda_W F\to\Lambda_W F
$$
\emph{is the completion map.} 

\pagebreak[3]
By the description of $\lambda^*_*$ preceding \cite[0.4.1]{AJL1}, this amounts to commutativity of the border of the following natural diagram:
\[
\def\1{$F$}
\def\2{$\kappa_*\kappa^*\<\< F$}
\def\3{$\kappa_*\kappa^*\Lambda_W F$}
\def\4{$\kappa_*\kappa^*\kappa_*F_{\</W}$}
\def\5{$\kappa_*F_{\</W}$}
\def\6{$\Lambda_W F$}
 \bpic[xscale=4, yscale=1.4]

   \node(11) at (1,-1){\1};   
   \node(12) at (2,-1){\6};   
   \node(13) at (3,-1){\5};
   
   \node(21) at (1,-2){\2};   
   \node(22) at (2,-2){\3};     
   \node(23) at (3,-2){\4};  
   
  %rows
    \draw[->] (11)--(12) node[above=1pt, midway, scale=.75] {$$};  
    \draw[double distance = 2pt] (12)--(13);
    
   \draw[->] (21)--(22) node[above=1pt, midway, scale=.75] {$$};  
    \draw[double distance = 2pt] (22)--(23);

 %columns 
     \draw[->] (11)--(21) ;
 
     \draw[->] (12)--(22) ;    
     
     \draw[<-] (2.97,-1.8)--(2.97,-1.2) ;
     \draw[->] (3.03,-1.8)--(3.03,-1.2) ;
    \epic
\]
Verification of the commutativity is left to the reader.
\end{proof}

\begin{cosa}\label{sigma,sigma}
Next we generalize Corollary~\ref{affine locimm}, replacing $u$ 
by an arbitrary flat map $f\colon X=\spec(S)\to\spec(R)=Z$ in $\sE$, corresponding to a flat
ring-homomorphism $\sigma\colon R\to S$. Lemma~\ref{L1.1.1} gives an expression for $\psi(f)$ for an \emph{arbitrary} flat $\sE$-map $f$, that in the foregoing affine case implies, as shown in Lemma~\ref{L1.1.3}, that for $M\in\D(R)$,
and $\env S\set S\otimes_R S$, $\psi(f)M$ is (naturally isomorphic to) the sheafification of the natural composite $\D(S)$-map
\begin{multline*}
\R\<\Hom_R(S,M)\iso
S\Otimes{\env S}\<(\env S\otimes_S\R\<\Hom_R(S,M))\\
\iso
S\Otimes{\env S}\<(S\otimes_R \R\<\Hom_R(S,M))
\lto\,
S\Otimes{\env S}\R\<\Hom_R(S,\>S\otimes_R M),
\end{multline*}
or, more simply, (see Proposition~\ref{R1.1.3.5}),
\[
\R\<\<\Hom_R(S,M) \to \R\<\<\Hom_R(S,S\otimes_R M)
\to 
S\Otimes{\env S}\R\<\<\Hom_R(S,S\otimes_R M).                                                                                                                                                                                                                                                                                                                                                                                                                                                                                                                                                        
\]
(The expanded notation in~\ref{L1.1.3} and~\ref{R1.1.3.5} indicates the $S$-actions involved.) 

So let $f:X\to Z$ be a 
\emph{flat} $\sE$-map, let $\delta:X\to X\times_ZX$ be
the diagonal, and let
$\pi_1^{},\pi_2^{}$  be the projections from \mbox{$X\times_ZX$} to $X$. 
There is a base-change isomorphism $\beta'=\pi_2^*f^!\iso\pi_1^! f^*\<$, 
as in~ \ref{bch}.
There is also a base-change map $\beta\colon \pi_2^*f^\times\to\pi_1^\times\<\< f^*\<$  as in Proposition~\ref{base change}\va1 (with $g=f,\; p=\pi_2,\; q=\pi_1$); 
this $\beta$ need not be an isomorphism.\va2

The next Lemma  concerns functors from $\Dqcpl(Z)$ to $\Dqcpl(X)$.

\begin{sublem}\label{L1.1.1} With preceding notation, there is an isomorphism of functors\/ $\nu\colon\LL\delta^*\pi_1^\times\<\< f^*\iso f^!$ 
such that the map\/ $\psi(f):f^\times\to f^!$ from~\textup{\ref{relation}} is the composite
\[
f^\times=\id_\sX^*\<\<f^\times\cong\LL\delta^*\pi_2^*f^\times\xto{\!\LL\delta^*\beta\>}
\LL\delta^*\pi_1^\times\<\< f^*\xto{\,\nu\,} f^!\<.
\]
\end{sublem}

\begin{proof}
Consider the diagram, where $\theta$ and $\theta'$ are the natural isomorphisms,
\[
\CD
f^\times @>\Iso>\theta>  \LL\delta^*\pi_2^*f^\times @>\LL\delta^*\beta>>
                                 \LL\delta^*\pi_1^\times\<\< f^* \\
@V\psi(f)VV @V\LL\delta^*\pi_2^*\psi(f)VV @VV\LL\delta^*\psi(\pi_1^{})V \\
f^!  @>\Iso>\theta'> \LL\delta^*\pi_2^*f^! @>>{\LL\delta^*\beta'}>
                                 \LL\delta^*\pi_1^!f^*
\endCD
\]
The left square obviously commutes, and the right square commutes by  Proposition~\ref{base change}. Since
$\pi_1\delta=\id_X$ is proper, Proposition~\ref{pullback1} guarantees that
$\LL\delta^*\psi(\pi_1^{})$ is an isomorphism, while $\LL\delta^*\beta'$
is an isomorphism since $\beta'$ is. 

The Lemma results, with $\nu\set(\theta')^{-\<1}\smallcirc(\LL\delta^*\beta')^{-\<1}
\smallcirc \LL\delta^*\psi(\pi_1^{}).$ \end{proof}

\begin{subcor}\label{R1.1.1.5}
The map\/ $\psi(f)$ in Lemma\/~\textup{\ref{L1.1.1}} factors as
\[
f^\times\! \xto{\,\eta\,} \R\>\pi_{2\>*}^{}\pi_2^*f^\times\! \xto{\R\>\pi_{2\>*}^{}\beta}
\R\>\pi_{2\>*}^{}\pi_1^\times\<\< f^*\xto{\,\eta\,} 
\R\>\pi_{2\>*}^{}\R\>\delta_*\LL\>\delta^*\pi_1^\times\<\< f^*\!\iso\! \LL\>\delta^*\pi_1^\times\<\< f^*\xto{\,\nu\,} f^!\<,
\]
where the maps labeled\/ $\eta$ are induced by units of adjunction, and the
isomorphism obtains because\/ $\pi_2\delta=\id_\sX$. 
\end{subcor}

\begin{proof}
By Lemma~\ref{L1.1.1} it suffices that the following diagram commute.
\[
\def\1{$f^\times$}
\def\2{$\R\>\pi_{2\>*}^{}\pi_2^*f^\times$}
\def\3{$\R\>\pi_{2\>*}^{}\pi_1^\times\<\< f^*$}
\def\4{$\LL\delta^*\pi_2^*f^\times$}
\def\5{$\R\>\pi_{2\>*}^{}\R\>\delta_*\LL\>\delta^*\pi_2^*f^\times$}
\def\6{$\R\>\pi_{2\>*}^{}\R\>\delta_*\LL\>\delta^*\pi_1^\times\<\< f^*$}
\def\7{$\LL\delta^*\pi_1^\times f^*$}
 \bpic[xscale=4.5, yscale=1.2]

   \node(11) at (1,-1){\1};    
   \node(12) at (2,-1){\2};   
   \node(13) at (3,-1){\3};   
 
   \node(21) at (1,-2){\4};   
   \node(22) at (2,-2){\5};  
   \node(23) at (3,-2){\6}; 
        
   \node(32) at (2,-3){\7};
   
  %rows
    \draw[->] (11)--(12) node[above, midway, scale=.75] {$\eta$};  
    \draw[->] (12)--(13) node[above, midway, scale=.75] {$\R\pi_{2\>*}\beta$};  
     
   \draw[->] (21)--(22) node[above=1pt, midway, scale=.75] {$$};  
   \draw[->] (22)--(23) node[above=1pt, midway, scale=.75] {$\via\beta$};    
      
 %columns 
    \draw[->] (11)--(21) node[left=1pt, midway, scale=.75] {$\simeq$};    
     
    \draw[->] (12)--(22) node[left=1pt, midway, scale=.75] {$\eta$}; 
              
    \draw[->] (13)--(23) node[right=1pt, midway, scale=.75] {$\eta$}; 

 %oblique
    \draw[->] (21)--(32) node[below=-1pt, midway, scale=.75] 
                                        {$\LL\delta^*\beta\mkern40mu$};  
    \draw[->] (32)--(23) node[below=1pt, midway, scale=.75] {$$};  

 %\labels 
 	   \node at (1.43,-1.51)[scale=.9]{\circled1};
  
  \epic
\]
Commutativity of the unlabeled subdiagrams is clear.

Subdiagram~\circled1
(without $f^\times$) expands as
\[
\def\1{$\id$}
\def\2{$\R\>\pi_{2\>*}^{}\pi_2^*$}
\def\4{$\LL\delta^*\pi_2^*$}
\def\5{$\R\>\pi_{2\>*}^{}\R\>\delta_*\LL\>\delta^*\pi_2^*$}
\def\8{$(\pi^{}_2\delta)_*\LL\delta^*\pi_2^*$}
\def\9{$(\pi^{}_2\delta)^*$}
\def\0{$(\pi^{}_2\delta)_*(\pi^{}_2\delta)^*$}
 \bpic[xscale=4, yscale=1.2]

   \node(11) at (1,-1){\1};    
   \node(13) at (3,-1){\2};   
  
   \node(21) at (1,-2){\9};   
   \node(22) at (2,-2){\0};  
   
   \node(31) at (1,-3){\4}; 
   \node(32) at (2,-3){\8};
   \node(33) at (3,-3){\5};
   
  %rows
   \draw[->] (11)--(13) node[above, midway, scale=.75] {$\eta$};  
      
   \draw[double distance=2pt] (21)--(22) ; 
    
   \draw[double distance=2pt] (31)--(32) ;    
   \draw[->] (32)--(33) node[above, midway, scale=.75] {$\Iso$};

 %columns 
    \draw[double distance=2pt] (11)--(21) ;  
      
    \draw[->] (21)--(31) node[left=1pt, midway, scale=.75] {$\simeq$}; 
    \draw[->] (22)--(32) node[right=1pt, midway, scale=.75] {$\simeq$}; 
              
    \draw[->] (13)--(33) node[right=1pt, midway, scale=.75] {$\eta$}; 

 %oblique
    \draw[->] (11)--(22) node[above=-2pt, midway, scale=.75] 
                                        {$\mkern20mu\eta$};  

 %\labels 
  	   \node at (2.5,-2)[scale=.9]{\circled2};
  
  \epic
\]
Commutativity of subdiagram \circled2 is given by \cite[(3.6.2)]{li}. Verification of   commutativity of the remaining two subdiagrams is left to the reader.
\end{proof}
 
\begin{subcosa}
\label{rem:affine-case}
We now concretize the preceding results in case  $X=\spec(S)$ and $Z=\spec(R)$ are affine, so that the flat map $f\colon X\to Z$  
corresponds to a flat  homomorphism $\sigma\colon R\to S$ of noetherian rings. 

First, some notation.
For a ring $P$,  $\M(P)$ will denote the category of $P$-modules.  Forgetting for the moment that $\sigma$ is flat, let  $\tau\colon R\to T$ be a flat ring-homomorphism. If
\[
\Hom_{\>\sigma\<,\tau}\colon \M(S)^\op\times \M(T)\to \M(T\otimes_R S)
\]
is the obvious functor such that  
\[
\Hom_{\>\sigma\<,\tau}(A,B)\set\Hom_R(A,B),
\] 
then, since (by flatness of $\tau$) any K-injective $T$-complex is K-injective over~$R$,  
there is a derived functor
\[
\R\<\<\Hom_{\>\sigma\<,\tau}\colon \D(S)^\op\times \D(T)\to \D(T\otimes_R S)
\]
such that, with $(F\to J_F\>)_{F\in\D(T)}$ a family of K-injective $T$-resolutions, and $E\in\D(S)$,
\[
\R\<\<\Hom_{\>\sigma\<,\tau}(E,F\>)\set\Hom_{\>\sigma\<,\tau}(E,J_F).
\] 

Set $\Hom_{\sigma}\set\Hom_{\sigma\<,\>\>\id_R}$.\va1
\end{subcosa}

Let $p^{}_1\colon T\to T\otimes_R S$ be the $R$-algebra homomorphism with $p^{}_1(t)=t\otimes 1$. There is a natural functorial isomorphism in $\D(T\otimes_R S)$:
\begin{equation}\label{ss to p}
\R\<\<\Hom_{\>\sigma\<,\>\tau}(E,\>F\>)\iso \R\<\<\Hom_{\>p^{}_1}\<\<(T\otimes_R E, F\>)
\quad(F\in\D(T)).
\end{equation}
(For this, just replace $F$ by a K-injective $T$-resolution.)

Let $p^{}_2\colon S\to T\otimes_R S$ be the $R$-algebra map with 
$p^{}_2(s)=1\otimes s$. Let $\rho_\tau\colon\D(T)\to\D(R)$ be the restriction-of-scalars functor induced by~ $\tau$; and define  $\rho_{p^{}_2}$ analogously. Then, in $\D(S)$,
\[
\R\<\<\Hom_{\>\sigma}(E, \rho_\tau F\>)=\rho_{p^{}_2}\R\<\<\Hom_{\>\sigma\<,\>\tau}(E, F\>)
\qquad(E\in\D(S),\,F\in\D(T)).
\] There results a ``multiplication" map
in $\D(T\otimes_RS)$:
\[
\mu\colon (T\otimes_R S)\otimes_S \R\<\<\Hom_{\>\sigma}(E, \rho_\tau F\>)\to\
\R\<\<\Hom_{\>\sigma\<,\>\tau}(E, F\>),
\]
and hence a natural composition in $\D(S)$
\begin{equation}\label{extend sigma}
\begin{aligned}
\R\<\<\Hom_{\>\sigma}(E, \rho_\tau F\>)&\iso 
S\Otimes{T\otimes_{\<R}\> S}\!\big((T\otimes_R S)\otimes_S \R\<\<\Hom_{\>\sigma}(E, \rho_\tau F\>)\big)\\
&\xto{S\Otimes{T\otimes_{\<R} S}\:\mu}
S\Otimes{T\otimes_{\<R}\> S}\R\<\<\Hom_{\>\sigma\<,\>\tau}(E,\>F\>).
\end{aligned}
\end{equation}

\vskip2pt
Now, assuming $\sigma$ to be flat, we derive algebraic expressions for $f^\times$ and~$f^!$. 

Application of the functor  $\R\Gamma(Z,-)=\R\<\<\Hom(\OZ,-)$ to item 1 in the proof of  Corollary~\ref{affine locimm}, 
gives $\R\<\<\Hom_Z(A^\sim\<,B^\sim)=\R\<\<\Hom_R(A,B)$. Since $(-)^{\sim_S}\colon\D(S)\to\Dqc(X)$ is an equivalence of categories \cite[p.\,230, 5.5]{BN},  it results from the canonical isomorphism (with $E\in\D(S)$,
$M\in\D(R)$ and $\sigma_{\<\<*}\colon\D(S)\to\D(R)$ the functor given by restricting scalars)
\[
\Hom_{\D(S)}\!\big(E, \R\<\<\Hom_\sigma(S,M)\big)\iso
\Hom_{\D(R)}(\sigma_*E,M)
\] 
that there is a functorial isomorphism
\begin{equation}\label{aff times}
\varrho(M)\colon \big(\R\<\<\Hom_\sigma(S,M)\big)^{\sim_S} \cong f^\times\<\< \big(M^{\sim_R}\big)
\qquad(M\in\D(R)) 
\end{equation}
such that $f_*\varrho(M)$ is the isomorphism $\zeta(M^{\sim_R})$ in Lemma~\ref{RHom}.\va1
 
Next, let $\pi_i\colon X\times_Z X\to X\ (i=1,2)$ be the projection maps, and let 
\mbox{$\delta\colon X\to X\times_Z X$} 
be the diagonal map.   Set $\env S\set S\otimes_R S$. Note that if $A\to B$ is a homomorphism of rings,  corresponding to
$g\colon\spec B\to\spec A$, and if $N\in\D(A)$, then
\begin{equation}\label{aff^*}
\LL g^*\big(N^{\sim_A}) =\big(B\Otimes{A} N\big)^{\sim_B}.
\end{equation}
This follows easily from the fact that the functor $(-)^{\sim_A}$ preserves both quasi-isomorphisms 
and K-flatness of complexes.

\begin{sublem}\label{L1.1.3}
There  is a natural functorial isomorphism of the map
\[
\psi(f)M^{\sim_R}:f^\times\< M^{\sim_R}\to f^!M^{\sim_R}
\qquad \big(M\in \D(R)\big)
\] 
with the sheafification of the natural composite\/ $\D(S)$-map
\begin{align*}
\psi(\sigma)M\colon \R\<\Hom_\sigma(S,M)
&\iso
S\Otimes{\env S}\<(\env S\otimes_S\R\<\Hom_\sigma(S,M))\\
&\iso
S\Otimes{\env S}\<(S\otimes_R \R\<\Hom_R(S,M))\\
&\,\lto\,
S\Otimes{\env S}\R\<\Hom_{\>\sigma\<,\>\sigma}(S,\>S\otimes_R M).
\end{align*}
\end{sublem}

\begin{proof}
Using  \eqref{aff times} and~\eqref{aff^*}, and the fact that sheafification is an equivalence of categories from $\D(S)$ to~
$\D(\spec S)$ (\cite[p.\,230, 5.5]{BN}), one translates the definition of the 
base\kf-change map~$\beta$ in~\ref{base change} to the commutative\kf-algebra context,  and finds that 
\[
\beta(M^{\sim_R}):\pi_2^*f^\times M^{\sim_R}\to\pi_1^\times\<\< f^*M^{\sim_R}
\] 
is naturally isomorphic to the sheafification of the natural 
composite $\D(\env S)$-map 
\[
S\otimes_R \R\<\<\Hom_\sigma(S,M) \to \R\<\<\Hom_{\>\sigma\<,\>\sigma}(S,S\otimes_R M)\iso\R\<\<\Hom_{\>p^{}_1}\<\<({\env S},S\otimes_R M)
\]
where the isomorphism comes from~\eqref{ss to p} (with $T=S$).

Lemma~\ref{L1.1.1} gives that $\psi(f)$ is naturally isomorphic to
the composite  
\[
f^\times\cong\LL\delta^*\pi_2^*f^\times \xto{{\LL\delta^*\beta}}
\LL\delta^*\pi_1^\times\<\< f^*\<,
\]
whence the conclusion.
\end{proof}

Here is a neater description of $\psi(\sigma)M$---and hence of $\psi(f)M^{\sim_R}$.

\begin{subprop}\label{R1.1.3.5}
The map\/ $\psi(\sigma)M$ in~\/\textup{~\ref{L1.1.3}\kf} factors as\va{-2}
\[\R\<\<\Hom_\sigma(S,M) 
\xto{\,\vartheta\,}\varpi
\R\<\<\Hom_\sigma(S,S\otimes_R M)
\xto{\!\eqref{extend sigma}\>} 
S\Otimes{\env S}\R\<\<\Hom_{\>\sigma\<,\>\sigma}(S,S\otimes_R M),                                                                                                                                                                                                                                                                                                                                                                                                                                                                                                                                                       
\]
where $\vartheta$ is  induced by the natural\/ $\D(R)$-map\/ 
$M\to S\otimes_R M\<$.
\end{subprop}

\begin{proof} 
Note that $\vartheta$ is the natural composite $\D(S)$-map\va{-2}
\[
\R\<\<\Hom_\sigma(S,M) \to S\otimes_R^{}\R\<\<\Hom_\sigma(S,M)
\to
\R\<\<\Hom_\sigma(S,S\otimes_R M),
\]
recall the description in the proof of ~\ref{L1.1.3} of the map $\beta$, refer to the factorization of\/ $\psi(f)M^{\sim_R}$ coming from\/~\textup{\ref{R1.1.1.5}}, and fill in the details. 
\end{proof}

From \ref{R1.1.3.5} and~\eqref{eta from psi.1} it follows easily that:

\begin{subcor} For any $N\in \D(S)$, the map $\eta(N^{\sim_S})$ from 
\textup{\ref{fake unit}} sheafifies
the natural composite $\D(S)$-map 
\begin{align*}
N\xto{\vartheta'}\Hom_\sigma(S,S\otimes_R N) 
&\lto\R\<\<\Hom_\sigma(S,S\otimes_R N)\\
&\xto{\!\eqref{extend sigma}\>} 
S\Otimes{\env S}\R\<\<\Hom_{\>\sigma\<,\>\sigma}(S,S\otimes_R N),                                                                                                                                                                                                                                                                                                                                                                                                                                                                                                                                                       
\end{align*}
where $\vartheta'$ takes $n\in N$ to the map $s\mapsto s\otimes n$. 
\hfill$\square$ 
\end{subcor}

Using Proposition~\ref{Gampsi}, we now develop more information about the above map $\psi(\sigma)M$ when $\sigma\colon k\to S$ is an 
essentially-finite\kf-type algebra over a field $k$, and $M= k$. 

For any $\fp\in\spec S$,
let\/ $I(\fp)$ be the injective hull of the residue field $\kappa(\fp)\set S_\fp/\fp S_\fp\>$.
Let $D^\sigma\in\D(S)$ be a \emph{normalized residual complex,} thus a complex
of the form \va{-1}
\[
D^\sigma:=\cdots0\to I^{-n}\to I^{-n+1}\to\cdots\to I\>^0 \to 0\cdots
\]
where for each integer $m$, $I^{-m}$ is the direct sum of the $I(\fp)$ as 
$\fp$ runs through the primes such that $S/\fp$ has dimension $m$. 
The sheafification of $D^\sigma$ is $f^!k,$ where $f\set\spec\sigma$ and where we identify $k$ with the structure sheaf of $\spec k$, see \cite[Chapter~VI, \S1]{RD}.

\begin{prop}\label{affine/k} Under the preceding circumstances, 
there exists a split exact sequence of\/ $S$-modules\va{-1}
\[
0\lto\bigoplus_{\makebox[5pt]{$\sst \fp\textup{ nonmaximal}$}}J(\fp) \lto
\Hom_\sigma(S,k) \xto{\,\psi^0\,} I\>^0\lto 0,
\]
such that for each nonmaximal prime\/ $\fp,$ $J(\fp)$ is a direct sum of uncountably many copies 
of~\/$I(\fp),$ and
in\/ $\D(S),$ $\psi(\sigma)k$ is the composition\va{-2} 
\[
\R\<\Hom_\sigma(S,k)=\Hom_\sigma(S,k)\xto{\psi^0} I\>^0
\hookrightarrow\D^\sigma.
\]
\end{prop}

\begin{proof} 
Since   $\Hom_{\sigma}(S,k)$ is an injective $S$-module,  there is a decomposition
\[
\Hom_{\sigma}(S,k) \cong \bigoplus_{\fp\>\in\>\spec S} I(\fp)^{\mu(\fp)}
\]
\vskip-2pt
\noindent where, $\sigma_\fp$ being the natural composite map
$k\xto{\sigma}S\twoheadrightarrow S/\fp$, $\mu(\fp)$ is the dimension of the $\kappa(\fp)$-vector space
\begin{align*}
\Hom_{S_\fp}\!\<\big(\kappa(\fp),\Hom_{k}(S,k)_\fp\big)
&=\Hom_S\!\big(S/\fp, \Hom_{\sigma}(S,k)\big)\otimes_S S_\fp\\
&\cong \Hom_{\sigma_\fp}(S/\fp,k)\otimes_{S/\fp}\kappa(\fp).
\end{align*}

In particular, if $\fp$ is maximal (so that $S/\fp=\kappa(\fp)$) then $\mu(\fp)=1$.
Thus $\Hom_{\sigma}(S,k)$ has a direct summand $J\>^0$ isomorphic to $I\>^0$. 
(This $J\>^0$ does not depend on the foregoing decomposition:
it consists of all $h\in\Hom_{\sigma}(S,k)$ such that the $S$-submodule $Sh$ has finite length.)

Now since $D^\sigma$ is a bounded injective complex, the $\D(S)$-map
$\psi(\sigma)$ is represented by an ordinary map of $S$-complexes 
$\Hom_{\sigma}(S,k)\to D^\sigma\<$, that is, by a map of $S$-modules
$\psi^0\colon \Hom_{\sigma}(S,k)\to I\>^0$.  By~\ref{L1.1.3}, the sheafification 
of~$\psi(\sigma)$ is 
$\psi(f)k\colon f^\times k\to f^!k$, and hence Proposition~\ref{Gampsi} implies that  $\psi^0$~maps~$J\>^0$ isomorphically onto $I\>^0$. Thus $\psi^0$ has a right inverse, unique up to
automorphisms of $I\>^0$; and $\Hom_{\sigma}(S,k)$ is the direct sum of $J\>^0$ and~ 
$\ker(\psi^0)$, whence
\[
\ker(\psi^0)\cong \bigoplus_{\makebox[7pt]{$\sst \fp\textup{ nonmaximal}$}}I(\fp)^{\mu(\fp)}.
\]
\vskip2pt
\noindent Last, in \cite[Theorem 1.11]{Nm3} it is shown that for nonmaximal~$\fp$,
\[
\mu_\fp=\dim_{\kappa(\fp)}\!\big(\< \Hom_{\sigma_\fp}\<(S/\fp,k)\otimes_{S/\fp}\kappa(\fp)\big)\ge|\#k|^{\aleph_0}, 
\]
with equality if $S$ is finitely generated over $k$. 
\end{proof}
\end{cosa}

\section{Applications}

\begin{cosa} (Reduction Theorems.) At least for flat maps, $\psi\colon(-)^\times\to(-)^!$ can be used to prove one of the main
results in \cite{AILN}, namely Theorem 4.6 (for which only a hint of a proof is given there). With notation as in \S\ref{sigma,sigma},  and again, $\env S\set S\otimes_R S$, that Theorem~4.6 asserts the~existence of a complex $D^\sigma\in\D(S)$, depending only on $\sigma\<$, and
for all $\sigma$-perfect $M\in\D(S)$ (i.e., $M$ is isomorphic in $\D(R)$ to a bounded complex of flat $R$-modules, the cohomology modules of $M$ are all finitely generated over $S$, and all but finitely many of them vanish), and all $N\in\D(S)$, a functorial $\D(S)$-isomorphism\va4 
\begin{equation}\label{AILN4.6}
\boxed{\R\<\<\Hom_S(M, D^\sigma)\Otimes{S} N \cong S\Otimes{S^e}\R\<\<\Hom_{\>\sigma\<,\>\sigma}(M,N).}
\end{equation}
In particular,
\begin{equation*}\label{affreldual}
D^\sigma\cong S\Otimes{S^e}\R\<\<\Hom_{\>\sigma\<,\>\sigma}(S,S).\tag*{(\ref{AILN4.6})$'$}
\end{equation*}
This explicit description is noteworthy in that  the sheafification \mbox{$\widetilde{D^\sigma}$}
is a \emph{relative dualizing complex} $f^!\OY$, where 
$f\set\spec\sigma\colon \spec S\to \spec R$ (see \cite[Example 2.3.2]{AIL}
or Lemma~\ref{L1.1.3} above); and otherwise-known definitions of $f^!$ involve choices, of which $f^!$ must be proved independent.\va2

The present proof will be based on the isomorphism in Lemma \eqref{Gmap} below,%
\footnote{After \cite{AILN} appeared, Leo Alonso and Ana Jerem\'ias informed us that Lemma  \eqref{Gmap} is an instance of \cite[p.\,123, (6.4.2)]{SGA5}---whose proof, however, is
not given in detail.}
which is similar to (and more or less implied by) the isomorphism in \cite[6.6]{AILN}.\va2

\
Let $f\colon X\to Z$ be an arbitrary map in $\sE$. Let $Y\set X\times_ZX$, and let $\pi_1$~
and~$\pi_2$ be the projections from $Y$ to $X$. 
For  $M, N\in\Dqc(X)$ there~are natural maps\va2
\begin{equation}
\label{setup}
\begin{aligned}
\pi_1^*\RHqc{\sX}(M,f^!\OZ)\Otimes{\<Y}\pi_2^*N
&\lto \RHqc{Y}(\pi_1^*M,\>\pi_1^*f^!\OZ)\Otimes{\<Y}\pi_2^*N\\
&\lto \RHqc{Y}(\pi_1^*M,\>\pi_1^*f^!\OZ\Otimes{\<Y}\pi_2^*N).
\end{aligned}
\end{equation}

\vskip1pt
\noindent The first of these is the unique one making the following otherwise natural diagram (whose top left entry is in $\Dqc(Y)$) commute:\va2
\begin{equation}\label{setup'}
\CD
\pi_1^*\RHqc{\sX}(M,f^!\OZ)\Otimes{\<Y}\pi_2^*N
@>>> \RHqc{Y}(\pi_1^*M,\>\pi_1^*f^!\OZ)\Otimes{\<Y}\pi_2^*N\\
@VVV @VVV\\
\pi_1^*\R\sHom_\sX(M,f^!\OZ)\Otimes{\<Y}\pi_2^*N
@>>> \R\sHom_Y(\pi_1^*M,\>\pi_1^*f^!\OZ)\Otimes{\<Y}\pi_2^*N
\endCD
\end{equation}
\vskip1pt

In \cite[\S5.7\kf]{AJL3} it is shown that for \emph{perfect} $\sE$-maps  $e\colon X\to Z$ (that~is, $e$~has finite flat dimension),
 the functor\va{-1} $e^!\colon\Dqcpl(Z)\to \Dqcpl(X)$ extends pseudo\-functorially to a 
functor---still denoted $e^!$---from $\Dqc(Z)$ to $\Dqc(X)$ such that \va1
\begin{equation}\label{extend!}
e^!F=e^!\OZ\Otimes{\sX}\LL e^*\<\<F\qquad (F\in\Dqc(X).
\end{equation}

For proper $e$, the extended $e^!$ is still right-adjoint to $\R e_*$ (see \cite[proof of Prop.\,5.9.3]{AJL3}).

The complex $M\in\D(X)$ is \emph{perfect relative to} $f$ (or simply $f$-\emph{perfect}) if $M$~has coherent cohomology and has
finite flat dimension over $Z$. In particular, the map~$f$ is perfect if and only if
$\OX$ is $f$-perfect.

\begin{sublem}
\label{Gmap}
If\/ the\/ $\sE$-map\/ $f\colon X\to Z$ is flat and\/ $M\in\D(X)$ is $f\<$-perfect, then for all\/ \mbox{$N\in\Dqc(X),$} the composite map\/~\eqref{setup} is an isomorphism.
\end{sublem}

\begin{proof} It holds  that
$\R\sHom_\sX(M,f^!\OZ)\in\Dqc(X)$ and 
$\R\sHom_Y(\pi_1^*M,\>\pi_1^*f^!\OZ)\in\Dqc(Y)$ (see proof of~\cite[6.6]{AILN}); and so the vertical 
arrows in \eqref{setup'} are isomorphisms. 
So is the bottom arrow in \eqref{setup'}
 (see e.g., \cite[ (4.6.6)]{li}). Hence the first map in \eqref{setup} is an
 isomorphism. 

As for the second, from the flatness of $f$ it follows that 
$\pi_1^*M$ is $\pi^{}_2$-perfect, and that there is a base\kf-change isomorphism (cf.~\eqref{bch})
\begin{equation}\label{bch2}
\pi_1^*f^!\OZ\iso\pi_2^!\OX.
\end{equation}
 The conclusion follows then from~\cite[6.6]{AILN} (with $g\set\pi^{}_2$, 
 $E'=\OX$, $F'=N$, and with $\R\sHom$ replaced throughout by $\RHqc{}$), in whose proof we can replace the duality isomorphism~(5.9.1) there by \eqref{sheaf dual} in this paper, and use the definition \eqref{extend!} of $e^!$ for any finite\kf-flat-dimensional map $e$ in~$\sE$ (for instance $g$, $h$ and $i$ in \emph{loc.\,cit.}), thereby rendering unnecessary the boundedness condition  in \emph{loc.\,cit.} on~the complex~$F'$.  (In this connection, note  that if $e=hi$ with $h$  smooth and $i$  a closed immersion then $i$ is perfect 
 \mbox{\cite[p.\,246, 3.6]{Il}}.)
\end{proof}

For $f\colon X\to Z$ a flat $\sE$-map and $M\in\Dqc(X)$  set \va{-1}
\[
M^\vee\set\RHqc{X}(M, f^!\OZ\<),
\]
and consider the composite map, with $N\in\Dqcpl(X),$
\begin{equation}
\label{setup2}
\RHqc{Y}(\pi_1^*M,\pi_2^{\times}\<N) 
\lto  \RHqc{Y}(\pi_1^*M,\pi_2^!N) \iso \pi_1^*M^{\vee} \Otimes{\<Y}\pi_2^*N
\end{equation}
where the first map is induced by $\psi(\<\pi^{}_2\<)$, and the isomorphism on the right is gotten by inverting the one given by \ref{Gmap} and then replacing 
\mbox{$\pi_1^*f^!\OZ\<\Otimes{\<Y}\<\pi_2^*N$} by the isomorphic object $\pi_2^!N$ (see \eqref{extend!} and \eqref{bch2}).
Remark 6.2 in \cite{AILN} authorizes replacement in (\ref{setup2}\kf) of $M$ by $M^\vee\<$, and recalls that the natural map is an isomorphism $M\iso \smash{{M^\vee}}^\vee$; thus one gets the composite map\va{-1}
\begin{equation*}
\label{setup2v}
\RHqc{Y}(\pi_1^*M^\vee\<,\pi_2^{\times}\<N) 
\lto  \RHqc{Y}(\pi_1^*M^\vee\<,\pi_2^!N) \!\iso\! \pi_1^*M \Otimes{\<Y}\pi_2^*N.
\tag*{(\ref{setup2}\kf){$^\vee$}}
\end{equation*}

\pagebreak

\begin{subthm}
\label{global4.6}
If\/ $M\in\Dqc(X)$ is\/ $f$-perfect, and\/ $N\in\Dqcpl(X),$\va{.6} then application of\/~ $\>\LL \delta^*\!\<$ 
$($resp.~ $\delta^{!})$ to the composite\/ \textup{(\ref{setup2}\kf)} $($resp.~\textup{(\ref{setup2}\kf)}$^\vee)$\va1 produces an isomorphism\va{-2}
\begin{align*}
\LL\delta^*\RHqc{Y}(\pi_1^*M,\>\pi_2^\times \<\<N) &\iso M^\vee\Otimes{\sX}N \\
(\kf resp.)\qquad\qquad\delta^!(\pi_1^*M\Otimes{Y}\pi_2^*N)&\osi  \RHqc{\sX}(M^{\vee},N) .
\end{align*}
\end{subthm}

\begin{proof}
By~\ref{pullback1}, application of $\LL\delta^*$ to the first map in~(\ref{setup2}\kf) produces an isomorphism.  Similarly, in view of \eqref{^timesHom}, applying  $\delta^\times$ ($=\delta^!\>$) to the first map 
in~(\ref{setup2}\kf)$^\vee$ produces an isomorphism. 
\end{proof}

%\enlargethispage*{10pt}
\begin{subrem} Using Remark~\ref{for fTd} 
for the first map in (\ref{setup2}\kf), \emph{one can extend  
Theorem~\ref{global4.6} to all} $N\in\Dqc(X)$. This results immediately
from the fact, given by \cite[Proposition 7.11]{Nm4}, that \emph{if\/ $e=pu$ is a compactification of a perfect\/ $\sE$-map\/ \mbox{$e\colon X\to Z$} then the following natural map is an isomorphism}:
\[
e^!N\!:\underset{\eqref{extend!}}{=\!\!=} e^!\OZ\Otimes{\sX}\LL e^*N
\cong u^*(p^\times\OZ\Otimes{\sX}\LL p^*\<\<N)\to u^*p^\times\<\< N\qquad(N\in\Dqc(Z)).
\]

\end{subrem}

\pagebreak[3]
\begin{subrem}\label{globalization}
The first isomorphism in Theorem~\ref{global4.6}  is a globalization (for flat  $f$
and cohomologically bounded-below $N$) of \cite[Theorem\,4.6]{AILN}.  \mbox{Indeed,} let $\sigma\colon R\to S$ be an essentially-finite-type flat homomorphism of noetherian rings,  $f=\spec(\sigma)$, $\env S\set S\otimes_R S$ and $p_i\colon S\to\env S\ (i=1,2)$ the canonical maps. 
Let   $M,N,D^\sigma\in\D(S)$, where $M$ is \mbox{$\sigma$-perfect} and $D^\sigma$ is~a relative dualizing complex, sheafifying to \mbox{$\widetilde{D^\sigma}=f^!\OZ$} \cite[Example 2.3.2]{AIL}.   Set $X\set\spec S$, $Z\set\spec R$, $Y\set X\times_ZX$, and let $\delta\colon X\to Y$ be the diagonal.
Then (as the cohomology of $M$ is bounded and finitely generated over $S$)
$\delta_*\<\big({\widetilde M}^\vee\Otimes{\sX}\widetilde N\big)$ sheafifies 
$\R\<\<\Hom_S(M,D^\sigma)\Otimes{S}N\in\D(\env S)$,
and, with notation as in \S\ref{sigma,sigma}, $\delta_*\LL\delta^*\R\>\sHom_Y(\pi_1^*M,\>\pi_2^\times \<\<N)$ sheafifies\va{-1}
\begin{align*} S\Otimes{{\env S}}\R\<\<\Hom_{\env S}(M\otimes_S \env S\<\<,\,\R\<\<\Hom_{\>p^{}_2}\<({\env S},N))
&\cong S\Otimes{{\env S}}\R\<\<\Hom_{\>p^{}_2}\<(M\otimes_S \env S,N)\\
&\cong S\Otimes{{\env S}}\R\<\<\Hom_{\>p^{}_2}\<(M\otimes_R S,N)\\
&\cong S\Otimes{{\env S}}\R\<\<\Hom_{\>\sigma\<,\>\sigma}(M,N).
\end{align*}
Thus in this situation, application of $\delta_{*}$ to \eqref{global4.6} gives the existence of a functorial isomorphism \eqref{AILN4.6}
(that should be closely related to---if not identical with---the one in \cite[4.6]{AILN}).
\end{subrem}

\begin{subrem}\label{redn-iso}
Let $f$ be as in \ref{Gmap}, and let $\delta\colon X\to Y\set X\times_Z X$ be the diagonal map. Keeping in mind the last paragraph of section~\ref{RHom^qc} above, one checks that the reduction isomorphism \cite[Corollary 6.5]{AILN} \va{-1}
\begin{equation*}\label{HomReduction}
\boxed{\delta^!(\pi_1^*M\Otimes{\sX}\pi_2^*N)\iso \R\>\sHom_\sX(M^\vee\<,\>N\big)}\\[-1pt]
\tag{\ref{redn-iso}.1}
\end{equation*}
is inverse to the second isomorphism in ~\ref{global4.6}. (In \cite{AILN}, see the proof of Corollary~6.5, and the last  four lines of the proof of Theorem~6.1 with $(X'\<,Y'\<,Y,Z)=(Y,X,Z,X)$, $E=\OY$, and
$(g,u,\<f,v)=(\pi^{}_2, f, f, \pi^{}_1)$, so that $\nu=\gamma=\id_\sX$.)

In the affine case, with assumptions on $\sigma$, $M$ and $N$ as above, ``desheafification" of (i.e., applying derived global sections to) \eqref{HomReduction} produces a functorial  isomorphism
\[
\boxed{\R\<\Hom_{\env S}(S, M\Otimes{R} N)\iso\R\<\Hom_S(\R\<\Hom_S(M,D^\sigma),N)}
\]
with the same source and target as the one in \cite[p.\,736, Theorem 1]{AILN}. (We suspect, but
don't know, that the two isomorphisms are the same---at least up to sign.)
\end{subrem}

\end{cosa}

\medskip

\begin{cosa}\label{intres}  
In this section we review, from the perspective afforded by results in this paper, some known basic facts about \emph{integrals, residues and fundamental classes.}  The description is mostly in abstract terms. What will be new is a direct \emph{concrete} description of the fundamental class of a flat 
essentially-finite\kf-type homomorphism $\sigma\colon R\to S$ of noetherian rings (Theorem~\ref{explicit fc}).\va2

Let $I\subset S$ be an ideal such that $S/I$ is a finite
$R$-module, and let $\Gamma^{}_{\!\<\<I}$  be the subfunctor of the identity functor on $S$-modules~$M$
given objectwise by
\[
\Gamma^{}_{\!\<\<I}(M)\set\{\,m\in M\mid I^nm=0\textup{ for some }n>0\,\}.
\]
There is an obvious map from the derived functor $\R\>\Gamma^{}_{\!\<\<I}$ to the identity functor on $\D(S)$.

In view of  the isomorphism~\eqref{aff times}, one can apply derived global sections in \mbox{Remark~\ref{barint} }to get, in the present context, the following diagram, whose rectangle commutes. In this diagram, 
$\sigma_{\<\<*}\colon\D(S)\to\D(R)$ is the functor given by restricting scalars; and $\omega_\sigma$ is a 
canonical module of~$\sigma$ (that is, an $S$-module whose sheafification is a relative dualizing sheaf of $f\set\spec\sigma$, as in Remark~\ref{barint}, where the integer $d$ is defined as well); and $D^\sigma$ is, as in ~\ref{globalization}, a relative dualizing complex.\va4
\[
\def\1{$\sigma_{\<\<*}\R\>\Gamma^{}_{\!\<\<I}\>\R\<\<\Hom_\sigma(S,R)$}
\def\2{$\sigma_{\<\<*}\R\>\Gamma^{}_{\!\<\<I}D^\sigma$}
\def\3{$\sigma_{\<\<*}\R\>\Gamma^{}_{\!\<\<I}\,\omega_\sigma[d]$}
\def\4{$\R\<\<\Hom_R(S,R)$}
\def\5{$\R\<\<\Hom_R(R,R)$}
\def\6{$R$}
\def\7{$\sigma_{\<\<*}\R\<\<\Hom_\sigma(S,R)$}
\bpic[xscale=4.7, yscale=2.5]

   \node(11) at (1,-1) {\1};     
   \node(12) at (2.3,-1){\2};   
   \node(13) at (2.92,-1){\3};   

   \node(151) at (1,-1.5){\7}; 

   \node(21) at (1,-2){\4};    
   \node(215) at (1.75,-2){\5};   
   \node(22) at (2.3,-2){\6};

 %rows
    \draw[->] (11)--(12) node[above, midway, scale=.7]{$\Iso$}
                                    node[below=1pt, midway, scale=.7]{$\via\psi$};
    \draw[<-] (12)--(13) ;

    \draw[->]  (21)--(215) node[below=1pt, midway, scale=.7]{$\via\sigma$};
    \draw[double distance=2pt] (215)--(22);
%columns 
    \draw[->] (11)--(151) ;
    \draw[double distance=2pt] (151)--(21);

    \draw[->] (12)--(22) node[right=2pt, midway, scale=.7]{$\bar{\!\int^{}}_{\!\!I}$} ;   

 \epic
\]
If $\sigma$ is Cohen-Macaulay and equidimensional, the natural map
is an isomorphism $\omega_\sigma[d\>]\iso D^\sigma$; and
application of $\textup{H}^0$ to the preceding diagram produces a commutative diagram of $R$-modules\va3
\[
\def\1{$\Gamma^{}_{\!\<\<I}\<\Hom_\sigma(S,R)$}
\def\2{$\Gamma^{}_{\!\<\<I}\>\>\sigma^!\<R$}
\def\3{$\textup{H}^d_I\,\omega_\sigma$}
\def\4{$\Hom_R(S,R)$}
\def\6{$R$}
\bpic[xscale=4.7, yscale=2.5]

   \node(11) at (1,-1) {\1};     
   \node(13) at (2.3,-1){\3};   

   \node(21) at (1,-2){\4};    
   \node(22) at (2.3,-2){\6}; 

 %rows
    \draw[->] (11)--(13) node[above, midway, scale=.7]{$\Iso$}
                                    node[below=1pt, midway, scale=.7]{$\via\psi$};

    \draw[->]  (21)--(22) node[below=1pt, midway, scale=.7]{$\textup{evaluation at 1}$};

%columns 
    \draw[->] (11)--(21) ;

    \draw[<-] (22)--(13) node[midway, right=1pt, scale=.7]{$\int_I$};

 \epic
\]
\vskip3pt
This  shows that \emph{an explicit description of\/ $(\via\psi)^{-1}$ \va1
is more or less the same as an explicit description of $\int_{\<I}$---and so, when $I$ is a maximal ideal, of residues.}  ``Explicit" includes the realization of
the relative canonical module~$\omega_\sigma$ in terms of regular differential forms (cf.~Remark~\ref{barint}).

Such a realization  comes out of the theory of the \emph{fundamental class}
$\fundamentalclass{f}$  of a flat $\sE$-map $f$, as indicated below. This $\fundamentalclass{f}$   is a key link between the abstract duality theory of $f$ and its canonical reification via differential forms. It may be viewed as an orientation, compatible with essentially 
\'etale base change, in a suitable bivariant theory on the category of flat $\sE$-maps \cite{AJL4}.\va2

Given a flat $\sE$-map $f\colon X\to Z$, with $\pi_1$ and $\pi_2$ the projections from $Y\set X\times_Z X$ to $X$, and $\delta\colon X\to Y$  the diagonal map,
let $\fundamentalclass{\<f}$ be, as in \cite[Example 2.3]{AJL4}, the natural composite $\D(X)$-map\va4
\begin{equation}\label{fclass}
\LL\delta^*\delta_*\OX\iso \LL\delta^*\delta_*\delta^!\pi_1^!\OX\lto\LL\delta^*\pi_1^!\OX\;
\underset{\eqref{bch}}\iso
%{\xto[\eqref{bch}]{\Iso}}\;
\LL\delta^*\pi_2^*f^!\OZ\iso f^!\OZ.
\end{equation}

Let $\mathcal J$ be the kernel of the natural surjection $\OY\twoheadrightarrow \delta_*\OX$.
Using a flat $\OY$-resolution of $\delta_*\OX$ one gets a natural isomorphism of $\OX$-modules
$$
\Omega^1_f=
\mathcal J/\>\mathcal J^{\>2}\cong \mathcal T\!or_1^{\OY}\!(\delta_*\OX,\delta_*\OX)=H^{-1}\LL\delta^*\delta_*\OX,
$$
whence a map of graded-commutative $\OX$-algebras, with $\Omega^i_f\set \wedge^{\!i}\mkern1.5mu \Omega^1_f\>$,
\begin{equation}\label{Omega and H}
\oplus_{i\ge 0} \,\Omega^i_f \to\oplus_{i\ge 0}\,\mathcal T\!or_i^{\OY}\!(\delta_*\OX,\delta_*\OX)
=\oplus_{i\ge 0}\,H^{-i}\LL\delta^*\delta_*\OX\>.
\end{equation}

In particular one has, with $d$ as above, a natural composition
\[
\gamma^{}_{\<f}\colon\Omega^d_f\to H^{-d}\LL\delta^*\delta_*\OX\xto{\<\<\via\>\fundamentalclass{\<f}}H^{-d}f^!\OZ=:\omega_{\<f}.
\]
(In the literature, the term ``fundamental class"  often refers to this $\gamma^{}_{\<f}$ rather than to 
$\fundamentalclass{f}\>$.) When $f$ is essentially smooth, this map is an isomorphism, as is 
$\omega_{\<f}\to f^!\OZ\>$. (The proof uses the known fact that there exists an isomorphism
$\Omega^d_f\iso f^!\OZ\>$, but does not reveal the relation between that isomorphism
and~$\gamma^{}_{\<f}$, see
\cite[2.4.2, 2.4.4]{AJL4}.) It follows that if $f$ is just \emph{generically} smooth, then 
$\gamma^{}_{\<f}$ is a generic isomorphism. For example, if $X$ is a reduced algebraic variety
over a field $k$, of pure dimension $d$, with structure map $f\colon X\to \spec k$, then one deduces that $\omega_{\<f}$ is canonically represented by
a coherent sheaf of meromorphic $d$-forms---the sheaf of regular $d$-forms---containing the sheaf~$\>\Omega^d_{\<\<f}$ of holomorphic $d$-forms, with equality over the smooth part of~$X\<$. 

From $\gamma^{}_{\<f}$ and the above $\int_{\<I}$ one deduces a map
\[
\textup{H}^d_I\>\Omega^d_\sigma\to R
\]
that generalizes the classical residue map.

\medskip
Theorem~\ref{explicit fc} below provides a direct concrete definition of the fundamental
class of a flat essentially-finite-type homomorphism $\sigma\colon R\to S$ of noetherian rings.

First, some preliminaries. As before, set $\env S\set S\otimes_R S$, let $p^{}_1\colon S\to\env S$ be the homomorphism such that for $s\in S$, $p^{}_1(s)=s\otimes 1$, and $p^{}_2\colon S\to\env S$ such that $p^{}_2(s)=1\otimes s$. 

Let $f\colon\spec S=:\<X\to Z\set\spec R$ be the scheme\kf-map corresponding to~$\sigma$. Let $\pi_1$ and~$\pi_2$ be the projections (corresponding to 
$p^{}_1$ and $p^{}_2$) from $X\times_Z X$ to $X\<$.\va1

Let $\Hom_{\>\sigma\<,\>\sigma}$ and $\Hom_{p^{}_1}$ be as in \S\ref{rem:affine-case}.

For an $S$-complex $F$, considered as an $\env S$-complex via the multiplication map $\env S\to S$, let  $\mu_F\colon F\to \Hom_{\>\sigma\<,\>\sigma}(S,\>F)$ be the $\env S$-homomorphism  taking\/ $f\in F$ to the map
$s\mapsto sf$.  

For an $S$-complex $E$, there is an obvious $\env S$-isomorphism 
\[
\Hom_{\>\sigma\<,\>\sigma}(S,\>E)\iso\Hom_{p^{}_1}\<\<(\env S,\>E). 
\]
Taking $E$ to be a K-injective resolution of $F$ (over $S$, and hence, 
since $\sigma$ is flat, also over $R$), one gets the isomorphism in the following statement.

\begin{sublem}\label{L3.2.3}
Let\/ $F\in\D(S)$ have sheafification $F^{\>\sim}\in\D(X)$.  The sheafification of the natural composite\/ $\D(\env S)$-map
\[
\xi(F)\colon F \xto{\,\mu_{\<F}\,}\Hom_{\>\sigma\<,\>\sigma}(S,\>F)
\lto \R\<\<\Hom_{\>\sigma\<,\>\sigma}(S,\>F)
\iso \R\<\<\Hom_{p^{}_1}\<(\env S\<,\>F)
\]
is the natural composite $(\<$with\/ $\epsilon^{}_2$ the counit map\/$)$
\begin{equation*}\label{iso323}
\delta_*F^{\>\sim}\iso\delta_*\delta^\times\<\pi_1^\times F^{\>\sim}
\xto{\ \epsilon^{}_2\ } 
\pi_1^\times F^{\>\sim}.
\tag{\ref{L3.2.3}.1}
\end{equation*}
\end{sublem}

\begin{proof}   The sheafification 
of $\R\<\<\Hom_{p^{}_1}\<\<(\env S\<,\>F)$ is $\pi_1^\times\<\< F^{\>\sim}$, 
see \eqref{aff times}. Likewise, with $m\colon \env S\to S$ the multiplication map, and $G\in\D(\env S)$, one has that $\delta^\times \<G^{\>\sim_{\env S}}$
is the sheafification of $\R\<\<\Hom_m(S, G)$; and Lemma~\ref{RHom} implies that $\epsilon^{}_2$ is the sheafification of the ``evaluation at 1"
map\[
\textup{ev}\colon \R\<\<\Hom_m(S,\R\<\<\Hom_{p^{}_1}\<\<(\env S\<,\>F\>))\to
\R\<\<\Hom_{p^{}_1}\<\<(\env S,\>F\>).
\]

Moreover, one checks that the isomorphism 
$\delta_*F^{\>\sim}\!\iso\<\<\delta_*\delta^\times\<\pi_1^\times F^{\>\sim}\,$ 
is the sheafification of the natural isomorphism\va1
$F\!\iso\<\<\R\<\<\Hom_m(S,\R\<\<\Hom_{p^{}_1}\!(\env S\<,\>F\>))$.

Under the allowable assumption that $F$ is K-injective, one finds then that 
\eqref{iso323} is the sheafification of the map 
$\xi'(F)\colon F \to\Hom_{p^{}_1}\<(\env S\<,\>F)$ that takes $f\in F$ to the map
$[s\otimes s'\mapsto ss'\<\<f\>]$. It is simple to check that $\xi'(F)=\xi(F)$.
\end{proof}

\begin{subthm}
\label{explicit fc} Let\/ $\sigma\colon R\to S$ be a flat essentially-finite-type map of noetherian rings, and\/ $f\colon\spec S\to\spec R$ the corresponding scheme-map. Let\/~$\mu\colon S\to  \Hom_{\>\sigma\<,\>\sigma}(S,S)$ be the $\env S$-homomorphism  taking\/ $s\in S$ to multiplication by\/ $s$. Then the fundamental class\/ $\Dc_{\<f}\<$ given by\/~\eqref{fclass} is naturally isomorphic to the sheafification of the natural composite map
\[
S\Otimes{\env S}S\xto{\<\id\otimes\mu\>} S\Otimes{\env S}\Hom_{\>\sigma\<,\>\sigma}(S,S)\lto 
S\Otimes{\env S}\R\<\<\Hom_{\>\sigma\<,\>\sigma}(S,S).
\]
\end{subthm}

\begin{proof} It suffices to show that the map in Theorem~\ref{explicit fc} sheafifies to a map isomorphic to the canonical composite map 
\begin{equation*}\label{fc eqn}
\LL\delta^*\delta_*\OX\iso \LL\delta^*\delta_*\delta^!\pi_1^!\OX\lto 
\LL\delta^*\pi_1^!\OX
\tag{\ref{explicit fc}.1}
\end{equation*}
(see \eqref{fclass}, in which the last two maps are isomorphisms). \va1

Applying pseudofunctoriality of $\psi$ (Corollary~\ref{relation}) to 
$\id_X=\pi_1\delta$, one sees that the  map in \eqref{fc eqn} factors as
\[
\LL\delta^*\delta_*\OX\to \LL\delta^*\delta_*\delta^\times\pi_1^\times\OX\to 
\LL\delta^*\pi_1^\times\OX\iso\LL\delta^*\pi_1^!\OX,
\]
where the isomorphism is from Proposition~\ref{pullback1}.
Thus the conclusion follows from Lemma~  \ref{L3.2.3}.
\end{proof}
\end{cosa}
 
\begin{subex} Let $T$ be a finite \'etale $R$-algebra. The (desheafified) fundamental class 
$\Dc_{R\to T}$ is the $\D(T)$-isomorphism from $T\Otimes{\env T} T =T$ 
to~$T\Otimes{\env T}\Hom_R(T,T)\cong\Hom_R(T,R)$ taking 1 to the trace~map. \va{.6}(Cf.~\cite[Example 2.6]{AJL4}.)\va1

If $S$ is an essentially \'etale 
$T$-algebra (for instance, a localization of $T$), then there is a canonical
identification of $\Dc_{R\to S}$ with $(\Dc_{R\to T})\otimes_{\>T} S$.
(This fact results from \cite[2.5 and 3.1]{AJL4}, but can be proved more directly.) 
However,~$\Dc_{R\to S}$ depends only on $R\to S$, not on $T$.

\end{subex}

\appendix

\section{Supports} 
\label{Support}
The goal of this appendix is to establish some basic facts---used repeatedly in ~\S\ref{proper-support}---about the relation between subsets of a noetherian scheme $X$ and ``localizing tensor ideals" in $\Dqc(X)$.\va2

\pagebreak[3]
 \emph{Notation}: Let $X$ be a noetherian scheme. For any $x\in X\<$,\va1 
 let $\sO_x$ be the stalk $\sO_{\<\<X\<\<,\>x\>}$,\va{.7}  let $\kappa(x)$ be the residue field of 
$\sO_x$, let $\widetilde{\kappa(x)}$ be the corresponding sheaf  on $X_x\set\spec \sO_x$---a~quasi-coherent, \emph{flasque} sheaf, \va{.4} let $\iota_x\colon X_x\to X$ be the canonical (flat) map---a localizing immersion,\va1 and let 
\[
k(x)\set \iota_{x*}\widetilde{\kappa(x)}=\R\iota_{x*}\widetilde{\kappa(x)},
\]
a quasi-coherent flasque $\OX$-module whose stalk at a point $y$ is $\kappa(x)$ if $y$ is a specialization of $x$, and 0 otherwise.\va1 

For $E\in\D(X)$, we consider two notions of the \emph{support of} $E\>$: 
\begin{align*}
\supp(E\>)&\set\{\,x\in X\mid E\Otimes{\sX} k(x)\ne 0\in\D(X)\,\},\\
\Supp(E\>)&\set\{\,x\in X\mid E_x\ne 0\in\D(\sO_x)\,\}.
\end{align*}

Let $\D_{\cc}(X)$ ($\Dcpl(X)$) be the full subcategory of $\D(X)$ spanned by the complexes with coherent cohomology modules (vanishing in all but finitely many negative degrees). For affine $X$ and $E\in\Dcpl(X)$ the next Lemma appears in ~\cite[top of page 158]{Fx}.  
\begin{lem}
\label{Supp} 
For any\/ $E\in\Dqc(X),$  
\[
\supp(E\>)\subseteq\Supp(E\>);
\]  
and equality holds whenever $E\in\D_{\cc}(X)$. 
\end{lem}

\begin{proof} 
For $E\in\Dqc(X)$, there is a  projection isomorphism\va{-2} 
\[
E\Otimes{\sX}k(x)\cong \R\iota_{x*}\<\big(\iota_x^*E\Otimes{X_x}\widetilde{\kappa(x)}\big).
\]
\vskip-1pt\noindent
Applying $\iota_x^*$ to this isomorphism, and  recalling from \S\ref{cosa:localizing-immersion} that $\iota_x^*\R\iota_{x*}$ is isomorphic to the identity, we get\va{-2}
\[
\iota_x^*\big(E\Otimes{\sX}k(x)\big)\cong \iota_x^*E\Otimes{\sX_{\<\<x}}\widetilde{\kappa(x)}.
\]
These two isomorphisms tell us that $E\Otimes{\sX}k(x)$ vanishes in $\Dqc(X)$ if and only if $\iota_x^*E\Otimes{\sX_{\<\<x}}\widetilde{\kappa(x)}$ vanishes in $\Dqc(X_x)$. 

Moreover, $\iota_x^*E\Otimes{\sX_{\<\<x}}\widetilde{\kappa(x)}$
is the sheafification of $E_x\Otimes{\sO_x}\kappa(x)\in\D(\sO_x)$, and so its vanishing in $\D(X_x)$ (i.e., its being exact) is equivalent to that of $E_x\Otimes{\sO_x}\kappa(x)$ in $\D(\sO_x)$. Thus\va{-3}
\[
x\in\supp(E) \iff E_x\Otimes{\sO_x}\kappa(x)\neq0.
\]

It follows that if $x\in\supp(E)$,  then $E_x\neq0$, that is to say, $x\in\Supp(E)$. So for all $E\in\Dqc(X)$ we have $\supp(E)\subseteq\Supp(E)$. \va2

%\enlargethispage*{4pt}
Now suppose $E\in\D_{\cc}(X)$ and $x\not\in\supp(E)$, i.e., $E_{x} \Otimes{\sO_{x}}\kappa(x)=0$.\va1 Let~$K$ be the Koszul complex on\va{.5} a finite set of generators for the maximal ideal of the local ring $\sO_{x}$. 
It is easy to check\va{.5} that the full subcategory of  $\D(\sO_{x})$ consisting of complexes $C$ such that $E_{x}\Otimes{\sO_{x}}C=0$ is a thick subcategory. It contains $\kappa(x)$, and hence also $K$,\va{.5} since the $\sO_{x}$-module $\oplus_{i\in\ZZ}\>\>\textup{H}^i(K)$
has finite length, see \cite[3.5]{DGI}. Thus\va1 $E_{x}\Otimes{\sO_{x}} K = 0$ in $\D(\sO_{x})$; and since the cohomology of $E_{x}$ is finitely generated in all degrees,\va{.5} \cite[1.3(2)]{FI}
gives~\mbox{$E_{x}=0$.}  Thus, $x\not\in\Supp(E)$;
and so $\supp(E)\supseteq\Supp(E)$.
\end{proof}

\pagebreak[3]
A \emph{localizing tensor ideal} $\sL\subseteq\Dqc(X)$ is a full triangulated subcategory of~$\,\Dqc(X)$, closed under arbitrary direct sums, and such that for all $G\in\sL$ and $E\in\Dqc(X)$, it holds that\va1
$G\Otimes{\sX}E\in\sL$.  

The next Proposition is proved in \cite[\S2]{Nm1} in the affine case; and in 
\cite{AJS} (where localizing tensor ideals are called \emph{rigid localizing subcategories}) the proof 
is extended to noetherian schemes. (Use e.g., \emph{ibid.,} Corollary~4.11 and the bijection in Theorem 4.12, as described at the beginning of its proof.)\va1

\begin{prop}
\label{ideals} 
Let\/ $\sL\subseteq\Dqc(X)$ be a localizing tensor ideal. A complex $E\in\Dqc(X)$ is in\/ $\sL$ if and only if so is $k(x)$ for all\/ $x$ in $\supp(E\>).$
\end{prop}

For closed subsets of affine schemes the next result is part of \cite [Proposition 6.5]{DG}.

\begin{prop}\label{Hom and tensor 0} 
Let\/  $E\in\Dqc(X)$ be such that $W:=\supp(E)$ is  a union of closed subsets of\/~$X\<$. 
\begin{itemize}
\item[\rm(i)]
For any\/ $F\in\Dqc(X),$\va{-1} 
\begin{align*}
E\Otimes{\sX}F=0&\iff\R\>\sHom_{\sX}(E,F)=0\\
&\iff \RHqc{\sX}(E,F)=0\\
&\iff\R\vG^{}_{\!W}F=0.
\end{align*}
\item[\rm(ii)]
For any morphism\/ $\phi\in\Dqc(X),$\va{-1} 
\begin{align*}
E\Otimes{\sX}\<\phi\textup{ is an isomorphism }
&\iff \R\>\sHom_{\sX}(E,\phi)\textup{ is an isomorphism }\\
&\iff \RHqc{\sX}(E,\phi)\textup{ is an isomorphism }\\
&\iff \R\vG^{}_{\!W}\phi\textup{ is an isomorphism.}
\end{align*}
\end{itemize}
\end{prop}

\begin{proof}
Let $\sL\subseteq\Dqc(X)$ (resp.~$\sL' \subseteq\Dqc(X))$ be the full subcategory spanned by the complexes $C$ such that $C\Otimes{\sX}F=0$ (resp.~$\R\>\sHom_{\sX}(C,F)=0$). It is clear that $\sL$ is a localizing tensor ideal;\va{.6} and using the natural isomorphisms (with $G\in\Dqc(X)$),\va{-2}
\[
\begin{aligned}\label{eqnA3}
\R\>\sHom_{\sX}\big(\!\oplus_{i\in I}C_i\>,\>F\big)&\cong\prod_{i\in I}\,\R\>\sHom_{\sX}(C_i\>,\>F),\\
\R\>\sHom_{\sX}(G\Otimes{\sX}C,F)&\cong\R\>\sHom_{\sX}(G,\R\>\sHom_{\sX}(C,F))\\[1pt]
\end{aligned}\tag{\ref{Hom and tensor 0}.1}
\]
\vskip-2pt\noindent
one sees that $\sL'$ is a localizing tensor ideal too.

We claim that when $E$ is in $\sL$ it is also in $\sL'$. For this it's enough,~by Proposition~\ref{ideals}, that for any $x\in W\<$, $k(x)$ be in $\sL'$.  By \cite[Lemma~3.4]{T},  there is a perfect $\OX$-complex $C$ such that $\Supp(C)$ is the  closure $\overline{\{x\}}$. \va{-1}We have\looseness=-1 
\[\supp(C)=\Supp(C)=\overline{\{x\}}\subseteq W, 
\]
where the first equality holds by Lemma~\ref{Supp} and the inclusion holds because $W$ is a union of closed sets. Thus \ref{ideals} yields $C\in\sL$; and the dual complex $C'\set\R\>\sHom_\sX(C,\OX)$ is in~$\sL'$, because $\R\>\sHom_\sX(C'\<,F)\cong C\Otimes{\sX} F=0$. Since\va{-1}
\[
x\in\supp(C) =\Supp(C)=\Supp(C')=\supp(C'),
\]
\vskip-1pt\noindent
therefore \ref{ideals} gives that, indeed, $k(x)\in\sL'\<$.
 
 \pagebreak[3]
Similarly, if $E\in\sL'$ then $E\in\sL$, proving the first part of (i). \va2

The same argument holds with $\RHqc{\sX}$ in place of $\R\>\sHom_{\sX}$. (After that replacement, the isomorphisms~\eqref{eqnA3} still hold if $\>\displaystyle\prod$ is prefixed
by $\id_X^\times$: this can be checked by applying the functors $\Hom_{\sX}(H,-)$
for all $H\in\Dqc(X)$.)\va2

As for the rest, recall that $\R\vG^{}_{\!W}\OX\in\Dqc(X)$: when $W$ itself is closed, this results from the standard triangle (with $w\colon X\setminus W\hookrightarrow X$ the inclusion) 
\[
\R\vG^{}_{\!W}\OX\to\OX\to \R w_*w^*\OX\xto{\,+\,}(\R\vG^{}_{\!W}\OX)[1]
\]
(or from the local representation of $\R\vG^{}_{\!W}\OX$ by a $\dirlm{}\!$ of Koszul complexes); and then for the general case,\va{-.5} use that $\vG^{}_W=\dirlm{}\,\vG^{}_{\!Z}$ where $Z$ runs through all closed subsets of $W\<$.  

By the following Lemma,  $\supp(\R\vG^{}_{\!W}\OX) = W\<$, so Proposition~\ref{ideals} implies that $E\in\sL$\va{.5} if and only if $\R\vG^{}_{\!W}\OX\in\sL$,
i.e., $E\Otimes{\sX}F=0$ if and only if $\R\vG^{}_{\!W}\OX\Otimes{\sX} F = 0$. 

\pagebreak
The last part of (i) results then from\va{.6} the standard isomorphism 
$\R\vG^{}_{\!W}\OX\Otimes{\sX} F \cong \R\vG^{}_{\!W}F$  (for which see, e.g., \cite[3.1.4(i) or 3.2.5(i)]{AJL1}\va1 when $W$ itself is closed, then pass to the general case using $\vG^{}_W=\dirlm{}\,\vG^{}_{\!Z}\>).$ 
And applying (i) to the third vertex of a triangle based on $\phi$ gives (ii).
\end{proof}

\begin{lem}
\label{suppRg}
 If\/ $W$ is a union of closed subsets of~
%the noetherian scheme\/
$X\<,$ then
\(
\supp(\R\vG^{}_{\!W}\OX) = W.\va{-3}
\)
\end{lem}

%\enlargethispage*{5pt}
\begin{proof} As seen a few lines back, $(\R\vG^{}_{\!W}\OX)\Otimes{\sX}k(x)\cong
\R\vG^{}_{\!W}k(x)$ for \mbox{any $x\in X\<$.}\va{.4} As $k(x)$ is flasque, the canonical map $\vG^{}_{\!W}k(x)\to\R\vG^{}_{\!W}k(x)$ is
an isomorphism.
The assertion is then that $\vG^{}_{\!W}k(x)\ne0\iff x\in W$ (i.e., $\overline{\{x\}}\subset W)$,\va{-.4} which is easily verified since $k(x)$ is constant on $\overline{\{x\}}$ and vanishes elsewhere. 
\end{proof}

\begin{lem}\label{L1.-1} 
Let\/ $u\colon W\to X$ be a localizing immersion, and\/ $F\in\Dqc(X)$.
The following conditions are equivalent.

{\rm (i)} $\supp(F)\subseteq W\<$.

 {\rm (ii)} The canonical map is an isomorphism\/ $F\iso\R u_*u^*\<\< F$.
 
 {\rm (iii)} $F\cong\R u_*G$ for some\/ $G\in\Dqc(W)$.
\end{lem}

\begin{proof} As in Remark~\ref{L1.1.25} , the canonical map 
$\R u_*G\to \R u_*u^*\R u_*G$ is an isomorphism,  whence (iii)$\Rightarrow$(ii); and the converse implication is trivial.\va1

Next,  if $x\notin W$ then $\overline{\{x\}}\cap W=\phi\>$: to see this, one reduces easily to the case in which  $u$ is the natural map $\spec A_M\to \spec A$,
where $M$ is a multiplicatively closed subset of the noetherian ring $A$ 
(see \S\ref{locimm}). Since $k(x)$~vanishes outside $\overline{\{x\}}$, it follows
that $u^*k(x)=0$ whenever $x\notin W$. Using the projection isomorphism 
\(
\R\>u_*G\Otimes{\sX} k(x)\cong \R\>u_*(G\Otimes{W}u^*k(x)),
\)
one~sees then that (iii)$\Rightarrow$(i).

The complexes $F\in\Dqc(X)$ satisfying (i) span a localizing tensor ideal. 
So do those $F$ satisfying (iii): the full subcategory 
\mbox{$\D_3\subseteq\Dqc(X)$} spanned by them is triangulated, as one finds by applying $\R u_*u^*$ to a triangle
based on a $\Dqc(X)$-map $\R u_*G_1\to\R u_*G_2\>$; $\D_3$ is closed under direct sums (since $\R u_*$ respects direct sums, see \cite[Lemma 1.4]{Nm2}, whose proof---in view~of the equivalence of categories mentioned above just before
\ref{affine locimm}---applies to $\Dqc(X))$;
and $\D_3$ is a tensor ideal since $\R u_*G\Otimes{\sX} E \cong \R u_*(G\Otimes{W}u^*\<\<E)$ for all
$E\in\Dqc(X)$.
So \ref{ideals} shows that for the implication (i)$\Rightarrow$(iii) we need only treat the case $F=k(x)$.

Since $\supp(k(x))=x$ (see, e.g., \cite[4.6, 4.7\kf]{AJS}), it suffices now to note 
that if $x\in W$ then 
\(
\sO_{W\<,\>x}=\sO_{\<\<X\<,\>x},
\) 
so the canonical map\va1 
\mbox{$\iota_x:W_x=X_x\to X$} in the definition of $k(x)$ (near the beginning of this Appendix) factors as 
$X_x\to W\xto{\lift.5,u,\>\>}X\<$, whence~$k(x)=\R\iota_{x*}\widetilde{\kappa(x)}$ satisfies (iii).  
\end{proof}

\begin{subrem}\label{supp u_*} With $u$ as in \ref{L1.-1}, one checks that if $x\in W$ then 
(with self-explanatory notation)  $u^*k(x)_\sX=k(x)_W$.
Also, as above, if~$x\notin W$ then
$u^*k(x)_\sX=0$. So for $E\in\Dqc(W)$,
\[
\R\>u_*E\Otimes{\sX}k(x)_\sX \cong \R\>u_*\big(E\Otimes{W} u^*k(x)_\sX\big) \cong
\begin{cases}
 0  & \text{if $x\notin W,$} \\
\R u_*\big(E\Otimes{W} k(x)_W\big) & \text{if $x\in W;$}
\end{cases}
\]
and since for $F\in\Dqc(W)$,  
$[0=F\cong u^*\R u_*F\>]\iff [\R u_*F=0\>]$, therefore
\[
\supp_\sX(\R\>u_*E)=\supp_W(E). \\[4pt]
\]

\end{subrem}

\pagebreak

\end{document}